\documentclass{amsart}

\newtheorem{theorem}{Theorem}[section]
\newtheorem{lemma}[theorem]{Lemma}
\newtheorem{proposition}[theorem]{Proposition}
\theoremstyle{definition}
\newtheorem{definition}[theorem]{Definition}
\newtheorem{example}[theorem]{Example}

\theoremstyle{remark}
\newtheorem{remark}[theorem]{Remark}

\usepackage{epsfig}

\reversemarginpar

\begin{document}

\newtheorem{conjecture}{Conjecture}[section]
\newtheorem{corollary}{Corollary}[section]
\newtheorem{Note}{Note}[section]

\theoremstyle{definition}
\newtheorem{rules}{Rule}[section]

\theoremstyle{remark}

\def\ve{\varepsilon}
\def\G{\Gamma}
\def\rar{\rightarrow}
\def\dis{\displaystyle}
\def\le{\left(}
\def\ri{\right)}

\def\res{\mathop{{\rm Res}}\limits}

%    Absolute value notation

\title{Mellin-Barnes integrals and the method of brackets}

\author[Ivan Gonzalez et al]{Ivan Gonzalez}
\address{Instituto de F\'{i}sica y Astronom\'{i}a, Universidad de Valparaiso, Gran Breta\~{n}a $1111$, Valparaiso, Chile}
\email{ivan.gonzalez@uv.cl}

\author[]{Igor Kondrashuk}
\address{Grupo de Matem\'{a}tica {A}plicada,~{D}epartamento de {C}iencias {B}\'{a}sicas,~{U}niversidad del 
{B}\'{i}o-{B}\'{i}o,~{C}ampus~{F}ernand~{M}ay, {C}hill\'{a}n,~{C}hile} 
\email{igor.kondrashuk@ubiobio.cl}

\author[]{Victor H. Moll}
\address{Department of Mathematics,
Tulane University, New Orleans, LA 70118}
\email{vhm@math.tulane.edu}

\author[]{Luis~M.~Recabarren}
\address{Departamento de F\'{i}sica, Universidad Santa Maria, Valparaiso, Chile}
\email{luis.recabarren@usm.cl}



\numberwithin{equation}{section}

\newcommand{\imagpart}{\mathop{\rm Im}\nolimits}
\newcommand{\realpart}{\mathop{\rm Re}\nolimits}
\newcommand{\no}{\noindent}
\newcommand{\ift}{\int_{0}^{\infty}}
\newcommand{\ione}{\int_{0}^{1}}
\newcommand{\eqf}{\stackrel{\bullet}{=}}
\newcommand{\op}[1]{\ensuremath{\operatorname{#1}}}
\newcommand{\pFq}[5]{\ensuremath{{}_{#1}F_{#2} \left( \genfrac{}{}{0pt}{}{#3}
{#4} \bigg| {#5} \right)}}

\begin{abstract}
The method of brackets is a method for the evaluation of definite integrals based on a small number of 
rules. This is employed here for the evaluation of Mellin-Barnes integral. The fundamental idea is 
to transform these integral representations into a bracket series to obtain their values.  The expansion of 
the gamma function in such a series constitute the main part of this new application. The power and flexibility of 
this procedure is illustrated with a variety of examples. 
\end{abstract}

\maketitle

\section{Introduction}
\label{sec-intro}

A classical approaches to special functions developed in the $19^{th}$ century was to classify them from the 
point of view of solutions 
 to (second) order differential equations with analytic coefficients
\begin{equation}
\label{ode-1}
a(z) \frac{d^{2}u}{dz^{2}} + b(z) \frac{du}{dz}  + c(z) u = 0.
\end{equation}
\noindent
Gauss understood that the singularities of this equation are central to this classification. Moreover, the role of 
$\infty$ as a possible singularity of \eqref{ode-1} plays an important role.  The notion of regular singularities and 
the well-known Frobenius method to produce solutions arose from this point of view. The reader will find 
information about these ideas in \cite{gray-2000a} and \cite{whittaker-2020a}. 

It is a classical result that any differential equation with $3$ regular singular points (on $\mathbb{C} \cup \{ \infty \}$) 
can be transformed to the hypergeometric differential equation
\begin{equation}
\label{hyper-ode1}
z(1-z) \frac{d^{2}w}{dz^{2}} + [ c - (a+b-1)z] \frac{dw}{dz} - a b w  = 0,
\end{equation}
\noindent
with singularities at $0, \, 1, \, \infty$.  The solution of \eqref{hyper-1}, normalized by $w(0) = 1$ is the hypergeometric 
function, defined by the series expansion 
\begin{equation}
\pFq21{a \,\,\,  b}{c}{z} = \sum_{n=0}^{\infty} \frac{(a)_{n} (b)_{n}}{(c)_{n} \, n!} z^{n},
\end{equation}
\noindent
where $(q)_{n}$ os the Pochhammer symbol  defined by 
\begin{equation}
(q)_{n} = \begin{cases} 1 & \textnormal{ if } n = 0 \\
q(q+1) \cdots (q+n-1) & \textnormal{ if } n > 0. 
\end{cases}
\end{equation}
\noindent
Among the many representations of the hypergeometric function we single out the so-called Mellin-Barnes integral 
\begin{equation}
\pFq21{a \,\,\,  b}{c}{z}  = \frac{\Gamma(c)}{\Gamma(a) \Gamma(b)} \times 
\frac{1}{2 \pi i } \int_{\gamma} \frac{\Gamma(a+s) \Gamma(b+s) \Gamma(-s)}{\Gamma(c+s)} (-z)^{s} \, ds,
\end{equation}
\noindent
where the contour of integration $\gamma$ joins $- i \infty$ to $+ i \infty$ and 
separates the poles of $\Gamma(-s)$ (namely $0, \, 1, \, 2, \ldots$) from those of 
$\Gamma(a+s) \Gamma(b+s)$ (located at $-a, \, -a-1, \ldots, -b, -b-1, \ldots$).  In the general case, the 
poles of the gamma factors are located on horizontal semi-axes. The 
contour $\gamma$ has to be chosen to separate those moving to the right from those moving to the left. 
Examples may be found in \cite{smirnov-2004a,smirnov-2006a}.

Many important functions are obtained from the hypergeometric function by coalescing singularities. For instance, the 
so-called Whittaker function  $W_{\kappa,\mu}(z)$, defined from the equation 
\begin{equation}
\frac{d^{2}w}{dz^{2}}+ \left( - \frac{1}{4} + \frac{k}{z} + \frac{1/4 - m^{2}}{z^{2}} \right) w = 0
\end{equation}
also has the integral representation 
\begin{equation}
W_{\kappa,\mu}(z) = \frac{e^{-z/2} }{2 \pi i } 
\int_{\gamma} z^{-s} 
\frac{\Gamma \left( \tfrac{1}{2} + \mu + s \right) \Gamma \left( \tfrac{1}{2} - \mu + s \right) 
\Gamma \left(- \kappa - s \right) }
{\Gamma \left( \tfrac{1}{2} + \mu - \kappa \right) \Gamma \left( \tfrac{1}{2} - \mu - \kappa \right) }  \, ds
\end{equation} 
\noindent
with a similar condition on the contour of integration. 

The main type of integrals considered in this work come from Mellin transforms. This concept is recalled next. 

\begin{definition}
Given a function $f(x)$, defined on the positive real axis $\mathbb{R}^{+}$, its Mellin transform is defined by 
\begin{equation}
\varphi(s) = (\mathcal{M}f)(s) = \int_{0}^{\infty} x^{s-1} f(x) \, dx.
\end{equation}
\noindent
This relation may be inverted by a line integral 
\begin{equation}
\label{mellin-2}
f(x) = (\mathcal{M}^{-1} \varphi)(x) = \frac{1}{2 \pi i } \int_{\gamma} 
x^{-s} \varphi(s) \, ds,
\end{equation}
\noindent
by an appropriate choice of the contour $\gamma$ of the type described above. 
\end{definition}

\smallskip

The goal of the current work is to present a procedure to evaluate inverse Mellin transforms based on the 
method of brackets. This is a method to evaluate definite integrals and it is described in Section 
\ref{sec-rules}. The class of functions considered here are of the form 
\begin{equation}
\label{gen-fun1}
\varphi(s) = 
 \frac{\prod\limits_{j=1}^{M}
\Gamma \left( a_{j}+A_{j}s\right) }{\prod\limits_{j=1}^{N}\Gamma \left(
b_{j}+B_{j}s\right) }\frac{\prod\limits_{j=1}^{P}\Gamma \left(
c_{j}-C_{j}s\right) }{\prod\limits_{j=1}^{Q}\Gamma \left(
d_{j}-D_{j}s\right) }.
\end{equation}
\noindent
For simplicity the  parameters 
$A_{j}, \, B_{j}, \, C_{j}, \, D_{j}$ are taken real and positive.  Integrals of the form \eqref{mellin-2} with an 
integrand of the type \eqref{gen-fun1} will be referred as 
\texttt{Mellin-Barnes integrals}. 

\begin{remark}
In  high energy physics  Mellin-Barnes integrals  appear frequently at the intermediate steps  in the process of calculations  
\cite{allendes-2010a,allendes-2012a,gonzalez-2017c,gonzalez-2013a,gonzalez-2013b,kniehl-2013a,
gonzalez-2018a,smirnov-2004a}. These are contour integrals 
in which the integrands are quotients of Euler Gamma functions, as in \eqref{gen-fun1}.  This
 type of contour integrals represent a typical representation of 
generalized hypergeometric functions. The Mellin-Barnes representation of denominators in integrands corresponding to Feynman diagrams in the momentum space representation  is a standard procedure  in 
quantum field theory. This approach helps to calculate very difficult Feynman diagrams \cite{smirnov-2004a} 
and this is the reason why  Barnes integrals usually appear at the penultimate step of calculation. 
\end{remark}

\begin{remark}
In the previous work \cite{Gonzalez:2015msa} it was  mentioned  that it would be interesting to combine the 
Mellin-Barnes integrals and the method of brackets, due to the similarity of these methods. 
This combination is presented here. In addition, several known relations 
for generalized hypergeometric functions have been reestablished by combining the  method of brackets and Mellin-Barnes transformations.  Our investigation has been motivated by the work 
of Prausa \cite{prausa-2017a} which proposes an interesting symbolic procedure. Our work compares it 
with the method of brackets.  
\end{remark}

 \section{The method of brackets}
 \label{sec-rules}

This is a method that 
evaluates definite integrals over the half line $[0, \, \infty)$. The 
application of the method consists of small number of rules, deduced in 
heuristic form, some of which are placed on solid 
ground \cite{amdeberhan-2012b}. 

For $a \in \mathbb{R}$, the symbol 
\begin{equation}
\langle a \rangle  = \ift x^{a-1} \, dx 
\end{equation}
is the {\em bracket} associated to the (divergent) integral on the right. The 
symbol 
\begin{equation}
\phi_{n} = \frac{(-1)^{n}}{\Gamma(n+1)} 
\end{equation}
\noindent
is called the {\em indicator} associated to the index $n$. The notation
$\phi_{i_{1}i_{2}\cdots i_{r}}$, or simply 
$\phi_{12 \cdots r}$, denotes the product 
$\phi_{i_{1}} \phi_{i_{2}} \cdots \phi_{i_{r}}$. 

\medskip

\noindent
{\bf {\em Rules for the production of bracket series}}

\smallskip

The first part of the method is to associate to the integral 
\begin{equation}
I(f) = \int_{0}^{\infty} f(x) \, dx 
\end{equation}
a bracket series according to 

\noindent
${\mathbf{Rule \, \, P_{1}}}$. Assume $f$ has the expansion 
\begin{equation}
f(x) = \sum_{n=0}^{\infty} \phi_{n} a_{n} x^{\alpha n + \beta -1 }.
\end{equation}
\noindent
Then $I(f)$ is assigned the \texttt{bracket series}
\begin{equation}
I(f)  = 
\sum_{n \geq 0}a_{n} \langle \alpha n + \beta \rangle.
\end{equation}

\smallskip

\noindent
${\mathbf{Rule \, \, P_{2}}}$. 
For $\alpha \in \mathbb{R}$, the multinomial power 
$(a_{1} + a_{2} + \cdots + a_{r})^{\alpha}$ is assigned the 
$r$-dimension bracket series 
\begin{equation}
\sum_{n_{1} \geq 0} \sum_{n_{2} \geq 0}  \cdots \sum_{n_{r} \geq 0}
\phi_{n_{1}\, n_{2} \,  \cdots n_{r}}
a_{1}^{n_{1}} \cdots a_{r}^{n_{r}} 
\frac{\langle -\alpha + n_{1} + \cdots + n_{r} \rangle}{\Gamma(-\alpha)}.
\end{equation}

\noindent
${\mathbf{Rule \, \, P_{3}}}$. 
Each representation of an integral by a bracket series has
associated a {\em complexity  index of the representation} via 
\begin{equation}
\text{complexity index } = \text{number of sums } - \text{ number of brackets}.
\end{equation}
\noindent 
It is important to observe that the complexity  index is attached to a specific 
representation of the integral and not just to integral itself.  The 
experience obtained by the authors using this method suggests that, among 
all representations of an integral as a bracket series, the one with 
{\em minimal complexity  index} should be chosen.  The level of difficulty in the analysis of the 
resulting bracket series increases with the complexity index. 

\medskip

\noindent
{\bf {\em Rules for the evaluation of a bracket series}}

\smallskip

\noindent
${\mathbf{Rule \, \, E_{1}}}$. 
Let $a, \, b \in \mathbb{R}$. The one-dimensional bracket series is assigned the value 
\begin{equation}
\label{mult-sum}
\sum_{n \geq 0} \phi_{n} f(n) \langle an + b \rangle = 
\frac{1}{|a|} f(n^{*}) \Gamma(-n^{*}),
\end{equation}
\noindent
where $n^{*}$ is obtained from the vanishing of the bracket; that is, $n^{*}$ 
solves $an+b = 0$. This is precisely the Ramanujan's Master Theorem.

\smallskip

The next rule provides a value for multi-dimensional bracket series of 
index $0$, that is, the number 
of sums is equal to the number of brackets. 

\smallskip

\noindent
${\mathbf{Rule \, \, E_{2}}}$. 
Let $a_{ij} \in \mathbb{R}$. Assuming the matrix
$A = (a_{ij})$ is non-singular, then the assignment is 
\begin{equation}
\sum_{n_{1} \geq 0} \cdots \sum_{n_{r} \geq 0} \phi_{n_{1} \cdots n_{r}} 
f(n_{1},\cdots,n_{r}) 
\langle a_{11}n_{1} + \cdots + a_{1r}n_{r} + c_{1} \rangle \cdots
\langle a_{r1}n_{1} + \cdots + a_{rr}n_{r} + c_{r} \rangle 
\nonumber
\end{equation}
\begin{equation}
=  \frac{1}{| \text{det}(A) |} f(n_{1}^{*}, \cdots n_{r}^{*}) 
\Gamma(-n_{1}^{*}) \cdots \Gamma(-n_{r}^{*}) 
\nonumber
\end{equation}
\noindent
where $\{ n_{i}^{*} \}$ 
is the (unique) solution of the linear system obtained from the vanishing of 
the brackets. There is no assignment if $A$ is singular. 

\smallskip

\noindent
${\mathbf{Rule \, \, E_{3}}}$. 
The value of a multi-dimensional bracket series of positive complexity  index
is obtained by computing all the contributions of maximal rank 
by Rule $E_{2}$. These contributions to the integral appear as series in the 
free indices. Series converging in a common region are added and divergent/nulls series are discarded.  There 
is no assignment to a bracket series of negative complexity index. If all the resulting series are 
discarded, then the method is not applicable. 

\begin{remark}
There is a small collection of formal operational rules for brackets. These will be used in the calculations presented 
below. 

\begin{rules}
For any $\alpha \in \mathbb{R}$ the bracket satisfies $\langle - \alpha \rangle = \langle \alpha \rangle$.
\end{rules}
\begin{proof}
This follows from the change of variables $x \mapsto 1/x$ in $\begin{displaystyle} \langle - \alpha \rangle = 
\int_{0}^{\infty} x^{-\alpha - 1} \, dx. \end{displaystyle}$
\end{proof}

A  similar change of variables gives the next scaling rule:

\begin{rules}
\label{rule2.2}
For any $\alpha, \, \beta, \, \gamma \in \mathbb{R}$ with $\alpha \neq 0$ the bracket satisfies 
\begin{equation}
\langle \alpha  \gamma + \beta \rangle = \frac{1}{| \alpha |} \left\langle  \gamma + \frac{\beta}{\alpha}  \right\rangle.
\end{equation}
\noindent
This can be deduced from Rule $E_{1}$. 
\end{rules}

\begin{rules}
For any $\alpha, \, \beta \in \mathbb{R}$ with $\alpha  \neq 0$, for any $n \in \mathbb{N}$ appearing as the 
index of a sum  any allowable function 
$F$, the identity 
\begin{equation}
F(n) \langle \alpha n + \beta \rangle = \frac{1}{| \alpha |} F \left( - \frac{\beta}{\alpha} \right) 
\left\langle  n  + \frac{\beta}{\alpha}  \right\rangle
\end{equation}
\noindent
in the sense that any appearance of the left-hand side in a bracket series may be replaced by the right-hand side.
\end{rules}
\begin{proof}
This follows directly from the rule $E_{1}$ to evaluate bracket series. 
\end{proof}

\end{remark}

\section{Some operational rules for integration}
\label{sec-operational}

This section describes the relation between line integrals, like those appearing in the inverse 
Mellin transform, and the 
method of brackets.  These complement those given for the discrete sums. 
The results presented here have appeared in \cite{prausa-2017a} as an extension 
of the method of 
brackets and used to produce a minimal Mellin-Barnes representations of integrals appearing in 
connection with Feynman diagrams. 

\subsection{The equivalence of brackets and the  Dirac's delta}
Let $f$ be a function defined on $\mathbb{R}^{+}$ and consider its Mellin transform 
\begin{equation}
F(s) = \int_{0}^{\infty} x^{s-1} f(x) \, dx,
\label{mellin-form1} 
\end{equation}
\noindent
with inversion rule 
\begin{equation}
f(x) = \frac{1}{2 \pi i } \int_{\gamma} x^{-s} F(s) \, ds,
\label{mellin-form2}
\end{equation}
\noindent
with the usual convention on the contour $\gamma$. Then, replacing \eqref{mellin-form2} into \eqref{mellin-form1} produces 
\begin{eqnarray}
F(s) & = & \int_{0}^{\infty} x^{s-1} \left[ \frac{1}{2 \pi i } \int_{\gamma} x^{-s'} F(s') \, ds' \right] \, dx \\ 
&  = & \frac{1}{2 \pi i } \int_{\gamma} F(s') \left[ \int_{0}^{\infty} x^{s - s' - 1} \, dx \right] \, ds' \nonumber \\
& = & \frac{1}{2 \pi i } \int_{\gamma} F(s') \langle s - s' \rangle \, ds'. \nonumber 
\end{eqnarray}

This proves:

\begin{rules}
\label{thm-fun1}
The rule for integration respect to brackets is given by 
\begin{equation}
\int_{\gamma} F(s) \langle s + \alpha \rangle \, ds = 2 \pi i F(-\alpha)
\end{equation}
\noindent 
where $\gamma$ is a contour of the usual type. The generalization 
\begin{equation}
\int_{\gamma} F(s) \langle \beta s + \alpha \rangle \, ds = \frac{2 \pi i }{| \beta| } 
F \left( - \frac{\alpha}{\beta} \right), \quad \textnormal{with} \,\, \beta \in \mathbb{R}
\end{equation}
\noindent 
can be obtained from Rule \ref{rule2.2}.
\end{rules}

\begin{remark}
From an operational point of view, the result in Theorem \ref{thm-fun1} may be written as 
\begin{equation}
\langle s + \alpha \rangle = 2 \pi i \delta(s+ \alpha),
\end{equation}
\noindent
where $\delta$ is Dirac's delta function.
\end{remark}

%\subsection{Relation to Kronecker's delta}
%The next operational result is to reinterpret the summation rule for brackets: from Rule $E_{1}$
%\begin{equation}
%\sum_{n \geq 0} \phi_{n} F(n) \langle n + \alpha \rangle = F(- \alpha) \Gamma(\alpha).
%\end{equation}

%\begin{proposition}
%The bracket satisfies the operational relation 
%\begin{equation}
%\langle n + \alpha \rangle = \frac{\Gamma(-n)}{\phi_{n}} \delta_{n, - \alpha},
%\end{equation}
%\noindent
%where 
%\begin{equation}
%\delta_{n,b} = \begin{cases}
%1 & \textnormal{ if } n = b \\
%0 & \textnormal{ if } n \neq b,
%\end{cases}
%\end{equation}
%\noindent
%is the Kronecker delta function.
%\end{proposition}

\subsection{Mellin transform} \label{MTB}

This section contains a brief  review of the Mellin transform. Recall that this is defined by 
\begin{equation}
 \label{MT}
\mathcal{M}(f)(z)  = \int_0^{\infty} x^{z-1}f(x)~dx, 
\end{equation}
where the arguments are the function $f$ to be transformed and the variable $z$ appearing in the integral.  The inverse Mellin transformation is 
\begin{equation}
 \label{MT Cauchy} 
f(x) = \frac{1}{2\pi i}\int_{c-i\infty}^{c + i\infty}x^{-z}\mathcal{M}(f)(z) \, dz \quad  \textnormal{ for } x \in (0,\infty).
\end{equation}
The point $c$ associated to the  contour of integration must be in the vertical strip $c_1 < c < c_2$, with 
boundaries determined by the condition that
\begin{eqnarray} \label{holom} 
\int_0^{1} x^{c_1-1}f(x)~dx  &  {\rm and } &  \int_1^{\infty} x^{c_2-1}f(x)~dx
\end{eqnarray}
must be finite.  This is satisfied if the function $f$ satisfies the growth conditions 
\begin{equation*}
|f(x)| < 1/x^{c_1} \quad {\rm when} \; x \to +0, \textnormal{ and }  \quad  |f(x)|< 1/x^{c_2} \quad
\textnormal{ when} \; x \to + \infty .
\end{equation*}
The conditions (\ref{holom}) imply that the Mellin transform $\mathcal{M}(f)(z)$ is holomorphic in 
the vertical strip $c_1 <  {\rm Re}~z  < c_2$.  The asymptotic behavior of the integrand is then used to 
determine the direction in which (a finite segment) of the vertical line contour is closed in order to produce 
a contour to apply Cauchy's integral theorem. The singularities of the integrand are then used to analyze the 
behavior of the integrals as the finite segment becomes infinite. 

%Should the contour in  Eq.(\ref{MT Cauchy}) be closed to the left complex infinity or to the right complex infinity  
%depends on the explicit asymptotic behaviour of the Mellin transform  $MT[f(x),x](z)$ at the complex infinity. 
%We close to the left if the left complex infinity does not contribute 
%and we close to the right if the right complex infinity does not contribute \footnote{In comparison, in the Mellin-Barnes transformation 
%we choose to which infinity the contour should be closed by taking into account 
%the absolute value of $x$ in  (\ref{MT Cauchy}) because the MB transform has already an established structure 
%in a form of fractions of the Euler $\G$ functions. However, MB transformation is only a particular case of Mellin transformation.}.
%Under this condition the original function $f(x)$ may be reproduced  via calculation of the residues by the Cauchy formula. 

One of the simplest examples of the Mellin transformation is
\begin{eqnarray*} 
\G(z) = \int_0^{\infty} e^{-x}x^{z-1}~dx & {\rm and }  &   e^{-x} = \frac{1}{2\pi i}\int_{c-i\infty}^{c + i\infty}x^{-z} \G(z) ~dz.
\end{eqnarray*}
The contour in the complex plane is the vertical line with ${\rm Re}~z = c$ is in the strip $0 < c < A,$ where $A$ is a real 
positive number, the vertical line contour must be closed to the left. It is convenient to include here a proof 
of equations (\ref{MT}-\ref{MT Cauchy}). This is classical and may be found  in any textbook on 
 the theory of complex variable. It is reproduced here for pedagogical arguments.   First,  use the fact 
  that  $\mathcal{M}(f)(z)$ is holomorphic in the strip $c_1 <  {\rm Re}~z  < c_2$, then taking 
  $\delta > 0$ to be infinitesimally small,
  % \footnote{$\delta$ in all this paper 
%stands for infinitesimally small real posiitve number.}
\begin{eqnarray}
& &  \label{MT direct} \\
\mathcal{M}(f) (z) & = & \int_0^{\infty} x^{z-1}f(x)~dx  \nonumber  \\
& = &  \frac{1}{2\pi i}\int_0^{\infty} x^{z-1} dx\int_{c-i\infty}^{c + i\infty}x^{-\omega} \mathcal{M}(f)(\omega) ~d\omega 
\nonumber  \\
& = & \frac{1}{2\pi i}\int_0^1 x^{z-1} dx\int_{c-i\infty}^{c + i\infty}x^{-\omega} \mathcal{M}(f)(\omega) ~d\omega   
\nonumber  \\
& & \quad \quad + \frac{1}{2\pi i}\int_1^\infty x^{z-1} dx\int_{c-i\infty}^{c + i\infty}x^{-\omega}
\mathcal{M}(f) (\omega) ~d\omega 
\nonumber \\
& = &  \frac{1}{2\pi i}\int_0^1 x^{z-1} dx\int_{c_1-\delta-i\infty}^{c_1-\delta + i\infty}x^{-\omega} 
\mathcal{M}(f)(\omega) ~d\omega \nonumber \\
& &  \quad \quad +  \frac{1}{2\pi i}\int_1^\infty x^{z-1} dx\int_{c_2+\delta-i\infty}^{c_2+\delta + i\infty}x^{-\omega} \mathcal{M}(f)(\omega) ~d\omega  \nonumber \\ 
& & = \frac{1}{2\pi i}\int_{c_1-\delta-i\infty}^{c_1-\delta + i\infty}\frac{\mathcal{M}(f)(\omega)}{z-\omega} ~d\omega  
- \frac{1}{2\pi i}\int_{c_2+\delta-i\infty}^{c_2+\delta + i\infty}\frac{\mathcal{M}(f)(\omega)}{z-\omega} ~d\omega
\nonumber  \\
& & = \frac{1}{2\pi i}\int_{c_1-\delta+i\infty}^{c_1-\delta - i\infty}\frac{\mathcal{M}(f)(\omega)}{\omega-z} ~d\omega + \frac{1}{2\pi i}\int_{c_2+\delta-i\infty}^{c_2+\delta + i\infty}\frac{\mathcal{M}(f)(\omega)}{\omega-z} ~d\omega 
\nonumber \\
& & = \frac{1}{2\pi i}\oint_{CR}\frac{\mathcal{M}(f)(\omega)}{\omega-z} ~d\omega, \nonumber 
\end{eqnarray}
\noindent
where $CR$ is a rectangular contour constructed from the two vertical lines from 
 (\ref{MT direct}) supplemented by two horizontal lines at the imaginary complex infinities of the
  strip $c_1 <  {\rm Re}~z  < c_2.$ The contour $CR$ is closed in the counterclockwise orientation. 

The inverse transformation proof is even simpler and may be used in order to define Dirac $\delta$-function.
Observe that 
\begin{eqnarray}
 \label{MT inverse} 
f(x) & = & \frac{1}{2\pi i}\int_{c-i\infty}^{c + i\infty}x^{-z} \mathcal{M}(f)(z) ~dz   \\
& = &  \frac{1}{2\pi i}\int_{c-i\infty}^{c + i\infty}x^{-z}~dz\int_0^{\infty} y^{z-1}f(y)~dy \nonumber \\
& = &  \int_0^{\infty}\delta(\ln{(y/x)}) y^{-1}f(y)~dy = f(x), \nonumber 
\end{eqnarray}
which  is valid in view of the relation 
\begin{eqnarray} 
\label{delta}
\frac{1}{2\pi i}  \int_{c - i\infty}^{c + i\infty}e^{(x-y)z}  dz & = & 
\frac{1}{2\pi }  \int_{-\infty}^{\infty}e^{(x-y)(c +i\tau)}  d\tau  \\
& = &
\frac{e^{(x-y)c}}{2\pi }\int_{-\infty}^{\infty}e^{i(x-y)\tau}  d\tau \nonumber \\
& = & e^{(x-y)c}\delta(x-y) = \delta(x-y). \nonumber
\end{eqnarray}

This proof of the inverse Mellin transformation belongs to D. Hilbert and may be found in any good textbook dedicated to complex analysis. In this paper we show 
that the method of bracket when applied to Mellin integrals  in some sense is equivalent to this old proof of the inverse Mellin transformation. {\it More precisely, we may argue that
by using the same trick like in this quite old proof of the inverse Mellin transformation, namely to divide the integration over the variable $x \in [0,\infty[$ in two parts,  
$x \in [0,1]$ and  $x \in [1,\infty[,$ and by creating in a such way a closed contour in the plane of the complex variable we may map all the rules of the method of bracket 
to the Cauchy integral formula.}

 In high energy particle physics, in order to solve  integro-differential equations representing evolution of important physical quantities, the transformation to Mellin moment 
is frequently used [see, for example \cite{DGLAP}]. 
The inverse transformation of the Mellin moment has completely the same form like the inverse Mellin transformation. The question may appear if the inverse Mellin transformation 
returns us to some function but how we may know if we came back to this function of a real variable from its Mellin moment or from its Mellin transform in the complex plane?  
We may differ the Mellin moments from the the Mellin transforms by studying the asymptotic behaviour at the complex infinity of the given function in the complex plane. 

\begin{remark}
The question of complexity, as introduced in \textbf{Rule} $\mathbf{P_{3}}$ is now extended. In the process of 
evaluating an integral by the method of brackets, define $\sigma$ to be the number of sums plus the number of 
contour integrals appearing and $\delta$ to be the number of brackets plus the number of integrals on the 
half-line $[0, \, \infty)$ that appear. The (generalized) index of complexity is $\iota = \sigma - \delta$. This 
index should be seen as a measure of difficulty in the evaluation of the integral by the method of brackets. In 
the case $\iota = 0$, the answer is given by a single term. For $\iota > 0$, the gamma factors appearing in the
numerators of the line integrals must be expanded in bracket series. This guarantees that the method provides
all series representations of the solution. As usual series converging in a common region must be added. It 
is the heuristic observation that the bracket/line integral representations of a problem should aim to minimize
the index $\iota$. 
\end{remark}

\subsection{Multiple integrals}
The method discussed here can be extended to evaluate multiple integrals with a bracket representation 
\begin{multline}
J = \left( \frac{1}{2 \pi i} \right)^{N} \int_{\gamma_{1}} ds_{1} \cdots \int_{\gamma_{N}}  ds_{N}  \\
F(s_{1}, \cdots, s_{N}) 
\langle a_{11}s_{1} + \cdots a_{1N}s_{N} + c_{1} \rangle \cdots 
\langle a_{N1}s_{1} + \cdots a_{NN}s_{N} + c_{N} \rangle,
\end{multline}
\noindent 
by using the one-dimensional rule in iterated form.  The expression for $J$ has the form 
\begin{equation}
J = \frac{1}{| \det(A) |} F(s_{1}^{*}, \cdots, s_{N}^{*})
\end{equation}
\noindent
where $A = (a_{ij})$, with $a_{ij} \in \mathbb{R}$ 
 and $\{ s_{i}^{*} \} \, \, (i=1, \ldots, N)$ is the solution of the system $A \vec{s} = - \vec{c}$ 
produced by the vanishing of the arguments in the brackets appearing in this process.

\section{Transforming Mellin-Barnes integrals to bracket series}
\label{sec-MB-to-brackets}

This section evaluates Mellin-Barnes integrals by transforming them into a bracket series. The rules of 
Section \ref{sec-rules} are then used to produce an analytic expression for the integral. 

\begin{lemma}
\label{gamma-brackets1}
The gamma function has the bracket series representation 
\begin{equation}
\Gamma(\alpha) = \sum_{n=0}^{\infty} \phi_{n} \langle \alpha + n \rangle.
\label{gamma-brack1}
\end{equation}
\end{lemma}
\begin{proof}
This follows simply 
 from  expanding  $e^{-t}$ in power series in the 
integral representation of the gamma function to obtain 
\begin{equation}
\Gamma(\alpha)  =  \int_{0}^{\infty} t^{\alpha-1} e^{-t} \, dt 
 =  \int_{0}^{\infty} t^{\alpha - 1} \sum_{n=0}^{\infty} \frac{(-1)^{n}}{n!} t^{n} \, dt 
 =  \sum_{n=0}^{\infty} \phi_{n} \langle \alpha + n \rangle.
\end{equation}
\end{proof}

In this section we present a systematic procedure to evaluate Mellin-Barnes integrals of the form 
\begin{equation}
\label{mellin-int0}
I(x) = \frac{1}{2 \pi i } \int_{\gamma} x^{-s} \frac{ 
\prod\limits_{j=1}^{M} \Gamma(a_{j} + A_{j}s) \prod\limits_{j=1}^{P} \Gamma(c_{j} - C_{j}s)
}
{
\prod\limits_{j=1}^{N} \Gamma(b_{j} + B_{j}s) \prod\limits_{j=1}^{Q} \Gamma(d_{j} - D_{j} s) 
} \, ds
\end{equation}
\noindent
where $\gamma$ is the usual  vertical line contour. A similar argument has appeared in \cite{prausa-2017a}.  The idea is
 to use the method of brackets to 
produce the bracket series associated to \eqref{mellin-int0}. The parameters satisfy the rules 
$a_{j}, \, b_{j}, \, c_{j}, \, d_{j} \in \mathbb{C}$ and $A_{j}, \, B_{j}, \, C_{j}, \, D_{j} \in \mathbb{R}^{+}$
for the index $j$ in the corresponding range; for instance, the $j$ associated to $a_{j}, \, A_{j}$ vary 
from $1$ to $M$. 

The procedure to obtain the bracket series is systematic: the gamma factors in the numerator 
are replaced using formula \eqref{gamma-brack1} (the gamma factors in the denominator do not 
contribute):
\begin{eqnarray}
\prod_{j=1}^{M} \Gamma(a_{j}+A_{j}s) & = & \prod_{j=1}^{M} \left[ \sum_{k_{j}} \phi_{k_{j}} \langle a_{j} + A_{j} s + k_{j}  \rangle \right] 
\label{gamma-brack2} \\
& = & \sum_{k_{1}} \cdots \sum_{k_{M}} \phi_{k_{1} \cdots k_{M}} \prod_{j=1}^{M} \langle a_{j} + A_{j} s + k_{j} \rangle  \nonumber 
\end{eqnarray}
\noindent
and similarly 
\begin{eqnarray}
\prod_{j=1}^{P} \Gamma(c_{j}-C_{j}s) & = & \prod_{j=1}^{P} \left[ \sum_{\ell_{j}} \phi_{\ell_{j}} \langle c_{j} - C_{j} s + \ell_{j} \rangle \right] 
\label{gamma-brack3} \\
& = & \sum_{\ell_{1}} \cdots \sum_{\ell_{P}} \phi_{\ell_{1} \cdots \ell_{P}} \prod_{j=1}^{P} \langle c_{j} - C_{j} s + \ell_{j} \rangle.  \nonumber 
\end{eqnarray}

The rules of the method of brackets described in Section \ref{sec-rules} now yield a bracket series associated with the integral 
\eqref{mellin-int0}.  To illustrate these idea, introduce the function 
\begin{equation}
G(s) = \frac{1}{\prod\limits_{j=1}^{N} \Gamma(b_{j} + B_{j}s)} \times  \frac{1}{\prod\limits_{j=1}^{Q} \Gamma(d_{j} - D_{j}s)},
\end{equation}
\noindent
The previous rules transform the integral $I(x)$ in \eqref{mellin-int0} to 
\begin{eqnarray}
I(x) & = & \frac{1}{2 \pi i } \sum_{k_{1}} \cdots \sum_{k_{M}} \sum_{\ell_{1}} \cdots \sum_{\ell_{P}} 
\phi_{k_{1} \cdots k_{M} \ell_{1} \cdots \ell_{P}} \\
&& \times \int_{\gamma} x^{-s} G(s) \left[ \prod_{j=1}^{M} \langle a_{j} + A_{j}s + k_{j} \rangle \right]
\,  \left[ \prod_{j=1}^{P} \langle c_{j} - C_{j}s + \ell_{j} \rangle  \right] \, ds. \nonumber 
\end{eqnarray}
this will be written in the more compact form 
\begin{equation}
I(x) = \frac{1}{2 \pi i } \sum_{\{ k \} } \sum_{\{ \ell \}} \phi_{\{k\}, \, \{\ell \}}  \int_{\gamma} x^{-s} G(s) 
\left[ \prod_{j=1}^{M} \langle a_{j} + A_{j}s + k_{j} \rangle \right]
\,  \left[ \prod_{j=1}^{P} \langle c_{j} - C_{j}s + \ell_{j} \rangle  \right] \, ds. \label{int-mellin1}
\end{equation}

Now select the bracket $\langle a_{M} + A_{M}s + k_{M} \rangle $ to evaluate the integral
 \eqref{int-mellin1}. Any other choice of bracket gives an equivalent  value for $I(x)$. Start with 
\begin{multline}
\label{int-mellin2} 
I(x) =  \sum_{\{ k \} } \sum_{\{ \ell \}} 
\phi_{ \{ k \}, \{ \ell \}}
\int_{\gamma} x^{-s} G(s)  \\ 
\left[ \prod_{j=1}^{M-1} \langle a_{j} + A_{j}s + k_{j} \rangle \right]
\,  \left[ \prod_{j=1}^{P} \langle c_{j} - C_{j}s + \ell_{j} \rangle  \right] \, 
\frac{\langle a_{M} + A_{M} s + k_{M} \rangle}{2 \pi i } \, ds. \nonumber 
\end{multline}
\noindent
The rules of the method of brackets given  requires to solve the linear equation coming from the vanishing of the last bracket and using Rule \ref{thm-fun1}.
This produces 
\begin{equation}
s^{*} = - \frac{a_{M} + k_{M}}{A_{M}}.
\end{equation}
\noindent
Therefore
\begin{multline}
\label{int-mellin3} 
I(x) = \frac{1}{|A_{M}|}  \sum_{\{ k \} } \sum_{\{ \ell \}} \phi_{ \{ k \}, \{ \ell \}}
 x^{-s^{*}} G(s^{*})  
%\left[ 
\prod_{j=1}^{M-1} \langle a_{j} + A_{j}s^{*} + k_{j} \rangle
% \right]
%\,  \left[ 
\prod_{j=1}^{P} \langle c_{j} - C_{j}s^{*} + \ell_{j} \rangle 
% \right]. 
\end{multline}
\noindent
Therefore, the value of the integral $I(x)$ obtained from the selection of the 
the bracket $\langle a_{M} + A_{M}s + k_{M} \rangle $ is given by 
\begin{multline}
I(x) = \frac{x^{a_{M}/A_{M}} }{|A_{M} |} \sum_{\{ k \}} \sum_{\{\ell \}} \phi_{\{ k \}, \{ \ell \}} x^{k_{M}/A_{M}} \\
\frac{
\prod\limits_{j=1}^{M-1} \left \langle a_{j} - \frac{A_{j} a_{M} }{A_{M} } - \frac{A_{j} k_{M} } {A_{M} }   + k_{j}  \right \rangle
\prod\limits_{j=1}^{P} \left \langle c_{j} + \frac{C_{j} a_{M} }{A_{M} } + \frac{C_{j} k_{M} } {A_{M} }   + \ell_{j}  \right \rangle
}
{
\prod\limits_{j=1}^{N} \left \langle b_{j} - \frac{B_{j} a_{M} }{A_{M} } - \frac{B_{j} k_{M} } {A_{M} }   \right \rangle
\prod\limits_{j=1}^{Q} \left \langle d_{j} + \frac{D_{j} a_{M} }{A_{M} } + \frac{D_{j} k_{M} } {A_{M} }   \right \rangle
}.
\end{multline}

Observe that one obtains a total of $M+P$ series representations for the integral $I(x)$. There are $P$ of 
them in the argument $x^{-1/C_{j}}$ and the remaining $M$ of them in the argument $x^{1/A_{j}}$. This 
procedure extends without difficulty to multiple integrals. 

\medskip

These ideas are  illustrated next.

\begin{example}
The hypergeometric function ${_{2}}F_{1}$, defined by the series 
\begin{equation}
\label{hyper-1}
\pFq21{a,b}{c}{x}  = \sum_{n=0}^{\infty} \frac{(a)_{n} (b)_{n}}{(c)_{n} \, n!} x^{n},
\end{equation}
for $|x| < 1$, admits the Mellin-Barnes representation 
\begin{equation}
\label{barnes-hyper1}
\pFq21{a,b}{c}{x} = \frac{\Gamma(c)}{\Gamma(a) \Gamma(b)} 
\frac{1}{2 \pi i } \int_{\gamma} \frac{\Gamma(-s) \Gamma(s+a) \Gamma(s+b)}{\Gamma(s+c)} 
(-x)^{s} \, ds
\end{equation}
\noindent
as a contour integral. This appears as entry $9.113$ in \cite{gradshteyn-2015a}.  

The integral in \eqref{barnes-hyper1} is now used to 
obtain the series representation \eqref{hyper-1}. As an added consequence, this 
method will also produce an analytic continuation of the series \eqref{hyper-1} to the domain $|x|>1$. 

The starting point is now the right-hand side of   \eqref{barnes-hyper1}
\begin{equation}
G(a,b,c;x) = 
\frac{\Gamma(c)}{\Gamma(a) \Gamma(b)} 
\frac{1}{2 \pi i } \int_{\gamma} \frac{\Gamma(-s) \Gamma(s+a) \Gamma(s+b)}{\Gamma(s+c)} 
(-x)^{s} \, ds
\end{equation}
\noindent
and using \eqref{gamma-brack1} in the three gamma factors yields
\begin{equation}
\label{form-G}
G(a,b,c;x) = \frac{\Gamma(c)}{\Gamma(a) \Gamma(b)} 
\frac{1}{2 \pi i } \sum_{n_{1},n_{2},n_{3}} \phi_{123} 
\int_{\gamma} 
\frac{(-x)^{s} \langle -s+n_{1} \rangle \langle s+ a + n_{2} \rangle \langle s+b + n_{3} \rangle}{\Gamma(s+c)} ds.
\end{equation}
\noindent
The gamma term in the denominator has no poles, so it is not expanded. 

In order to evaluate the expression \eqref{form-G}, select the bracket containing the index $n_{1}$ 
 and use Theorem
\ref{thm-fun1} to obtain the  bracket series:
\begin{equation}
\label{value-G}
G(a,b,c;x) = \frac{\Gamma(c)}{\Gamma(a) \Gamma(b)} \sum_{n_{1},n_{2},n_{3}} 
\phi_{123} \frac{(-x)^{n_{1}}}{\Gamma(n_{1}+c)} \langle a + n_{1} + n_{2} \rangle 
\langle b+n_{1}+n_{3} \rangle.
\end{equation}
\noindent
The evaluation of this series is done according to the rules given in Section \ref{sec-rules}. 

\smallskip 

\noindent
\texttt{Take $n_{1}$  as the free index}.  Then the indices $n_{2}, \, n_{3}$ are determined by the system 
\begin{equation}
a+n_{1}+n_{2} = 0 \quad \text{ and } \quad b+n_{1}+n_{3} = 0,
\end{equation}
\noindent
which  gives $n_{2} = - a - n_{1}$ and $n_{3} = -b-n_{1}$. Then \eqref{mult-sum} produces 
\begin{equation}
G_{1}(a,b,c;x) = \frac{\Gamma(c)}{\Gamma(a) \Gamma(b)} \sum_{n_{1}=0}^{\infty} \phi_{1}
\frac{(-x)^{n_{1}}}{\Gamma(n_{1}+c)} \Gamma(a+n_{1}) \Gamma(b+n_{1}),
\end{equation}
\noindent
where the index on $G_{1}$ is used to indicate that this sum comes from the free index $n_{1}$. This reduces
 to \eqref{hyper-1}, showing that 
\begin{equation}
G_{1}(a,b,c;x) =   \sum_{n=0}^{\infty} \frac{(a)_{n} (b)_{n}}{(c)_{n} \, n!} z^{n}.
\end{equation}
\noindent 
The series on the right converges for $|x|<1$.  This recovers equation \eqref{hyper-1}. 

\smallskip 

\noindent
\texttt{Take $n_{2}$  as the free index}.  Then the vanishing of the brackets give $n_{1} = -a - n_{2}$ and 
$n_{3} = -b+a+n_{2}$. Then 
\begin{equation}
\label{series-3}
G_{2}(a,b,c;x) = \frac{\Gamma(c)}{\Gamma(a) \Gamma(b)} 
\sum_{n_{2}=0}^{\infty} \phi_{2} \frac{(-x)^{-a-n_{2}}}{\Gamma(-a-n_{2}+c)}
\Gamma(a+n_{2}) \Gamma(b-a-n_{2}).
\end{equation}
\noindent
Using $\Gamma(u+m) = \Gamma(u) (u)_{m}$ converts \eqref{series-3} into 
\begin{equation}
\label{series-4}
G_{2}(a,b,c;x) = \frac{\Gamma(c) \Gamma(b-a)}{ \Gamma(b) \Gamma(c-a)} 
\sum_{n_{2}=0}^{\infty} \phi_{2} \frac{(-x)^{-a-n_{2}}}{ (c-a)_{-n_{2}}} 
(a)_{n_{2}} (b-a)_{-n_{2}}.
\end{equation}
\noindent
The final step uses the transformation rule 
\begin{equation}
(u)_{-n}  = \frac{(-1)^{n}}{(1-u)_{n}}
\end{equation}
\noindent
to eliminate the negative indices on the Pochhammer symbols and convert \eqref{series-4} into 
\begin{eqnarray}
\label{series-5}
\quad G_{2}(a,b,c;x) & = &  \frac{\Gamma(c) \Gamma(b-a)}{ \Gamma(b) \Gamma(c-a)} 
\sum_{n_{2}=0}^{\infty} \phi_{2} \frac{(-x)^{-a-n_{2}}}{ (1-b+a)_{n_{2}}} 
(a)_{n_{2} } (1-c+a)_{n_{2}} \\
& = & (-x)^{-a}  \frac{\Gamma(c) \Gamma(b-a)}{\Gamma(b) \Gamma(c-a)} 
\sum_{n_{2}=0}^{\infty} \frac{(a)_{n_{2}} (1-c+a)_{n_{2}}}{(1-b+a)_{n_{2}} \, n_{2}!} x^{-n_{2}}. \nonumber
\end{eqnarray}
\noindent
The series on the right is identified as a hypergeometric series and it yields 
\begin{equation}
G_{2}(a,b,c;x) = (-x)^{-a} \frac{\Gamma(c) \Gamma(b-a)}{\Gamma(b) \Gamma(c-a)} 
\pFq21{a \,\,\,\, \,\, 1-c+a}{1-b+a}{\frac{1}{x}}
\end{equation}
\noindent
and this series converges for $|x|>1$.

\smallskip 

\noindent
\texttt{Finally take  $n_{3}$  as the free index}. This case is similar to the previous one and it produces 
\begin{equation}
G_{3}(a,b,c;x) = (-x)^{-b} \frac{\Gamma(c) \Gamma(a-b)}{\Gamma(a) \Gamma(c-b)} 
\pFq21{b \,\,\,\,\,\, 1-c+b}{1-a+b}{\frac{1}{x}}
\end{equation}
\noindent
and this series also converges for $|x|>1$.

The rules in Section \ref{sec-rules} state that if in the evaluation of an integral  one obtains a collection of 
series,  coming from choices of free indices, those converging in a common region must be added. Thus, the 
integral $G(a,b,c;x)$ in \eqref{value-G} has the representations
\begin{equation}
G(a,b,c;x)  = \pFq21{a \,\, b}{c}{x} \quad \text{for} \,\, |x| < 1 
\end{equation}
\noindent
and 
\begin{multline}
G(a,b,c;x) = 
 (-x)^{-a} \frac{\Gamma(c) \Gamma(b-a)}{\Gamma(b) \Gamma(c-a)} 
\pFq21{a \,\,\,\, \,\, 1-c+a}{1-b+a}{\frac{1}{x}}  \label{analytic1} \\ + 
(-x)^{-b} \frac{\Gamma(c) \Gamma(a-b)}{\Gamma(a) \Gamma(c-b)} 
\pFq21{b \,\,\,\,\,\, 1-c+b}{1-a+b}{\frac{1}{x}}, \quad \text{for} \,\, |x|>1.
\end{multline}

Therefore we have obtained an analytic continuation of the hypergeometric function $_{2}F_{1}(a,b,c;x)$ from 
$|x|<1$ to the exterior of the unit circle. The identity \eqref{analytic1} 
appears as entry $9.132.2$ in \cite{gradshteyn-2015a}. 
\end{example}

\section{Inverse Mellin transforms}
\label{sec-MB-IMT}

The method of brackets is now used to evaluate integrals of the form \eqref{mellin-2}
\begin{equation}
f(x) = \frac{1}{2 \pi i } \int_{\gamma} x^{-s} \varphi(s) \, ds,
\end{equation}
\noindent
where $\gamma$ is a vertical line on $\mathbb{C}$, adjusted to each problem. Given $\varphi(s)$, the function
 $f(x)$ has Mellin transform $\varphi$.

\begin{example}
\label{example-1}
Consider the function  $\varphi(s) = \Gamma(s-a).$ Its inverse Mellin transform is given by
\begin{equation}
\label{ex-1}
f(x)= \frac{1}{2 \pi i} \int_{\gamma}  x^{-s} \Gamma(s-a) \, ds.
\end{equation}
\noindent
Now use \eqref{gamma-brack1} to write 
\begin{equation}
\Gamma(s-a) = \sum_{n} \phi_{n} \langle s - a + n \rangle
\end{equation}
\noindent
and \eqref{ex-1} yields 
\begin{equation}
f(x) = \sum_{n} \phi_{n} \frac{1}{2 \pi i } 
\int_{\gamma} x^{-s} \langle s-a+n \rangle \, ds
\end{equation}
\noindent
Theorem \ref{thm-fun1} now gives 
\begin{equation}
f(x)= \sum_{n} \phi_{n} x^{n-a} = x^{-a}e^{-x}.
\end{equation}
\noindent
This is written as 
\begin{equation}
\frac{1}{2 \pi i } \int_{\gamma}  x^{-s} \Gamma(s-a) \, ds = x^{-a} e^{-x},
\end{equation}
\noindent
or equivalently 
\begin{equation}
\int_{0}^{\infty} x^{s-a-1}e^{-x} \, dx = \Gamma(s-a).
\end{equation}
\noindent
Replacing $s-a$ by $s$, this is the integral definition of the gamma function.
\end{example}

\begin{example}
\label{example-1a}
The inverse Mellin transform of $\varphi(s) = \Gamma(a-s)$ is obtained as in Example \ref{example-1}. The result is 
\begin{equation}
f(x) = x^{-a} e^{-1/x},
\end{equation}
\noindent
also written as 
\begin{equation}
\frac{1}{2 \pi i } \int_{\gamma} x^{-s} \Gamma(a-s) \, ds = x^{-a} e^{-1/x},
\end{equation}
\noindent
or equivalently
\begin{equation}
\int_{0}^{\infty} x^{s-1}x^{-a}e^{-1/x} \, dx = \Gamma(a-s).
\end{equation}
\noindent
The change of variables  $u = x^{-1}$ gives the integral representation of the gamma function.
\end{example}

\begin{example}
\label{example-2}
The  inversion of $\varphi(s) = \Gamma(s-a) \Gamma(s-b)$ amounts to the 
evaluation of the line integral
\begin{equation}
\label{ex-2}
f(x)= \mathcal{M}^{-1} (\Gamma(s-a) \Gamma(s-b))(x) = 
\frac{1}{2 \pi i} \int_{\gamma} x^{-s} \Gamma(s-a) \Gamma(s-b) \, ds.
\end{equation}
\noindent
Now use \eqref{gamma-brack1} to write 
\begin{equation}
\Gamma(s-a) = \sum_{n_{1}} \phi_{n_{1}} \langle s - a + n_{1} \rangle \quad \text{and} \quad
\Gamma(s-b) = \sum_{n_{2}} \phi_{n_{2}} \langle s - b + n_{2} \rangle 
\end{equation}
\noindent
and produce
\begin{equation}
\label{series-n1}
f(x) = \frac{1}{2 \pi  i} \int_{\gamma} x^{-s} \sum_{n_{1},n_{2}} \phi_{12} 
\langle s-a+n_{1} \rangle \langle s - b + n_{2} \rangle.
\end{equation}
\noindent
Now select the bracket containing the index  $n_{1}$ and write \eqref{series-n1} using
 Theorem \ref{thm-fun1} as 
\begin{eqnarray}
f(x) & = &   \sum_{n_{1},n_{2}} \phi_{12} \frac{1}{2 \pi i } \int_{\gamma}  \left( x^{-s} \langle s - b + n_{2} \rangle \right) 
\langle s - a + n_{1} \rangle \\
& = & \sum_{n_{1}, n_{2} } \phi_{12} \, x^{-a+n_{1} } \langle a-n_{1} - b + n_{2} \rangle. \nonumber
\end{eqnarray}
\noindent
This is a two-dimensional bracket series and its evaluation is achieved using the rules in 
Section \ref{sec-rules}: \\

\noindent
\texttt{$n_{1}$ is a free index}. Then $n_{2} = n_{1}-a+b$ and this produces  the value 
\begin{eqnarray}
f_{1}(x) & = & \sum_{n_{1}=0}^{\infty} \phi_{1} x^{-a+n_{1}} \Gamma(-n_{1}+a-b) \\
& = & 
x^{-a}  \Gamma(a-b) \sum_{n_{1}=0}^{\infty} \phi_{1} (a-b)_{-n_{1}}  \nonumber \\
& = & x^{-a} \Gamma(a-b) \sum_{n_{1}=0}^{\infty} \frac{x^{n_{1}}}{n_{1}! \, (1-a+b)_{n_{1}}} \nonumber \\
& = & x^{-a} \Gamma(a-b)\, \pFq01{-}{1-a+b}{x}. \nonumber 
\end{eqnarray}

\noindent
\texttt{$n_{2}$ is a free index}.  A similar argument gives 
\begin{equation}
f_{2}(x) = x^{-b} \Gamma(b-a) \, \pFq01{-}{1-b+a}{x}. 
\end{equation}

Since both representations always converge, one obtains 
\begin{equation}
\label{two-IBessel}
f(x) =  x^{-a} \Gamma(a-b)\, \pFq01{-}{1-a+b}{x} +  x^{-b} \Gamma(b-a) \, \pFq01{-}{1-b+a}{x}. 
\end{equation}

The function $_{0}F_{1}$ is now expressed in terms of the modified Bessel function $I_{\nu}(z)$. This is 
defined in \cite[10.25.2]{olver-2010a} by the power series
\begin{equation}
\label{mod-bes1}
I_{\nu}(z) = \left( \frac{z}{2} \right)^{\nu} \sum_{k=0}^{\infty} \frac{1}
{k! \, \Gamma(\nu+k+1)} \left( \frac{z^{2}}{4} \right)^{k}.
\end{equation}

\begin{lemma}
\label{lemma-1}
For $\alpha \in \mathbb{R}$, the identity 
\begin{equation} 
 \pFq01{-}{\alpha}{x}  = \Gamma(\alpha) x^{(1-\alpha)/2} I_{\alpha -1}(2 \sqrt{x})
 \end{equation}
 \noindent
 holds.
 \end{lemma}
 \begin{proof}
 This follows directly from \eqref{mod-bes1}.
 \end{proof}
 
 Replacing the expression in Lemma \ref{lemma-1}  in \eqref{two-IBessel} gives 
 \begin{equation}
 f(x) = \frac{\pi}{\sin(\pi (a-b))} x^{-(a+b)/2} \left( I_{b-a}(2 \sqrt{x}) - I_{a-b}(2 \sqrt{x}) \right).
 \end{equation}
 \noindent
 The relation \cite[10.27.4]{olver-2010a}
 \begin{equation}
 K_{\nu}(z) = \frac{\pi}{2} \frac{I_{-\nu}(z) - I_{\nu}(z)}{\sin( \pi \nu))} 
 \end{equation}
 \noindent
 (which is  here as the definition of $K_{\nu}(z)$) now  implies 
 \begin{equation}
 f(x) = 2 x^{-\nu/2 - b} K_{\nu}(2 \sqrt{x});
 \end{equation}
 \noindent
 with $\nu = a-b$.  This is 
 \begin{equation}
 \label{mellin-ex1}
 \frac{1}{2 \pi i } \int_{\gamma} x^{-s} \Gamma(s-a) \Gamma(s-b) \, ds = 
 2x^{-(a+b)/2} K_{a-b}(2 \sqrt{x}).
 \end{equation}
 \noindent
 After some elementary changes, this is written as 
 \begin{equation}
\label{mellin-bes1}
K_{\nu}(x) = \frac{1}{4 \pi i } \left( \frac{x}{2} \right)^{\nu} 
\int_{\gamma} \Gamma(s) \Gamma(s- \nu) \left( \frac{x}{2} \right)^{-2s} \, ds,
\end{equation}
\noindent
the form appearing in  \cite[10.32.13]{olver-2010a}.  The expression \eqref{mellin-bes1} is 
now written in the equivalent form 
\begin{equation}
\label{bes-int-1}
\int_{0}^{\infty}x^{s-1} K_{\nu}(x) \, dx = 2^{s-2} \Gamma \left( \frac{s+\nu}{2} \right) 
\Gamma \left( \frac{s-\nu}{2} \right).
\end{equation}
 \end{example}

\begin{example}
\label{example-3}
Consider the  inversion of $\varphi(s) = \Gamma(s-a) \Gamma(b-s)$. Observe that there is a change 
in the order of the argument of the second gamma factor with respect to Example \ref{example-2}. To evaluate 
this example, expand the term $\Gamma(s-a)$ in a bracket series using \eqref{gamma-brack1} to obtain 
\begin{equation}
f(x) = \sum_{n} \phi_{n} \frac{1}{2 \pi i } \int_{\gamma} x^{-s} \langle s - a + n \rangle \Gamma(b-s) \, ds.
\end{equation}
\noindent
Theorem \ref{thm-fun1} yields 
\begin{equation}
f(x) = x^{-a} \sum_{n=0}^{\infty} \phi_{n} \Gamma(b-a+n) x^{n}.
\end{equation}
\noindent
To simplify this answer write $\Gamma(b-a+n)= \Gamma(b-a) (b-a)_{n}$, use 
\begin{equation}
\pFq10{\alpha}{-}{x} = (1-x)^{-\alpha}
\end{equation}
\noindent
and conclude that 
\begin{equation}
f(x) = \frac{\Gamma(b-a)}{x^{a}(1+x)^{b-a}}.
\end{equation}
\noindent
This is equivalent to the evaluation 
\begin{equation}
\frac{1}{2 \pi i } \int_{\gamma} x^{-s} \Gamma(s-a) \Gamma(b-s) \, ds = 
\frac{\Gamma(b-a)}{x^{a}(1+x)^{b-a}}.
\end{equation}
\noindent
Expanding the other gamma factors produces the same analytic expression for the integral.
\end{example}

\begin{example}
\label{example-4}
This example considers the simplest case of an integrand where a quotient of gamma factors appears. This is the 
inversion of 
\begin{equation}
\varphi(s)  = \frac{\Gamma(s-a)}{\Gamma(s-b)}.
\end{equation}
\noindent
The usual formulation now gives 
\begin{equation}
f(x) = \sum_{n} \phi_{n} \frac{1}{2 \pi i } \int_{\gamma} \left( \frac{x^{-s}}{\Gamma(s-b)} \right) 
\langle s-a+n \rangle \, ds.
\end{equation}
\noindent
Theorem \ref{thm-fun1} now yields 
\begin{equation}
f(x) = \sum_{n=0}^{\infty} \phi_{n} \frac{x^{n-a}}{\Gamma(a-b-n)}.
\end{equation}
\noindent
This expression is simplified using 
\begin{equation}
\Gamma(a-b-n) = \Gamma(a-b) (a-b)_{-n} = (-1)^{n}  \frac{\Gamma(a-b)}{(1-a+b)_{n}}
\end{equation}
\noindent
to obtain 
\begin{eqnarray}
f(x) & = & \frac{x^{-a}}{\Gamma(a-b)} \sum_{n=0}^{\infty} \frac{(1-a+b)_{n}}{n!} x^{n} \\
& = &  \frac{1}{\Gamma(a-b)} x^{-a} (1-x)^{-1+a-b}. \nonumber
\end{eqnarray}
\noindent
This can be written as 
\begin{equation}
\frac{1}{2 \pi i } \int_{\gamma} x^{-s} \frac{\Gamma(s-a)}{\Gamma(s-b)} \, ds = 
\frac{1}{\Gamma(a-b)} x^{-a}(1-x)^{-1+a-b}.
\end{equation}
\end{example}

\begin{example}
\label{example-5}
The inverse Mellin transform $f(x)$ of 
\begin{equation}
\varphi(s) = \frac{\Gamma(s-a)}{\Gamma(b-s)}
\end{equation}
\noindent
is computed from the line integral
\begin{equation}
f(x) = \frac{1}{2 \pi i } \int_{\gamma} x^{-s}  \frac{\Gamma(s-a)}{\Gamma(b-s)}  \, ds.
\end{equation}
\noindent
The usual procedure now yields 
\begin{equation}
f(x) = \sum_{n=0}^{\infty}  (-1)^{n} \frac{x^{n-a}}{n! \, \Gamma(b-a+n)}  = 
\frac{x^{-a}}{\Gamma(b-a)} \sum_{n=0}^{\infty} \frac{(-1)^{n}}{n! \, (b-a)_{n}} x^{n}.
\end{equation}
\noindent
The series is identified as an ${_{0}F_{1}}$ and using the identity 
\begin{equation}
J_{\nu}(z) = \frac{1}{\Gamma(\nu+1)} \left( \frac{z}{2} \right)^{\nu} 
\pFq01{-}{\nu+1}{- \frac{z^{2}}{4}}  
\end{equation}
\noindent
produces 
\begin{equation}
f(x) = x^{(1-a-b)/2} J_{-1-a+b}( 2 \sqrt{x})
\end{equation}
\noindent
and gives the evaluation
\begin{equation}
\frac{1}{2 \pi i } \int_{\gamma} x^{-s}  \frac{\Gamma(s-a)}{\Gamma(b-s)}  \, ds = x^{(1-a-b)/2} J_{-1-a+b}( 2 \sqrt{x}).
\end{equation}
\end{example}

\begin{example}
\label{example-6}
The inverse Mellin transform $f(x)$ of
\begin{equation}
\varphi(s) = \frac{\Gamma(a-s)}{\Gamma(s-b)}
\end{equation}
\noindent
is computed as in Example \ref{example-5}. The result is 
\begin{equation}
\label{bessel-mellin1}
\frac{1}{2 \pi i } \int_{\gamma} x^{-s} \frac{\Gamma(a-s)}{\Gamma(s-b)} \, ds = 
x^{-(a+b+1)/2} J_{a-b-1} \left( \frac{2}{\sqrt{x}} \right)
\end{equation}
\noindent
and thus 
\begin{equation}
\int_{0}^{\infty} x^{s-1} x^{-(a+b+1)/2} J_{a-b-1}  \left( \frac{2}{\sqrt{x}} \right) \, dx = 
\frac{\Gamma(a-s)}{\Gamma(s-b)}.
\end{equation}
\noindent
The identity \eqref{bessel-mellin1}, with $a=0$ and $b=-\nu-1$,  can be written as 
\begin{equation}
\label{mellin-bes3}
J_{\nu}(x) = \frac{1}{2 \pi i } \int_{\gamma} \frac{\Gamma(-s)}{\Gamma(s+ \nu +1)} 
\left( \frac{x}{2} \right)^{2s + \nu} \, ds.
\end{equation}
\end{example}

\begin{example}
The Mellin  inversion of the function 
\begin{equation}
\varphi(s) = \frac{\Gamma(s) \Gamma(1-s)}{\Gamma(\beta - \alpha s)}
\end{equation}
\noindent
is given by 
\begin{equation}
f(x) =  \frac{1}{2 \pi i} \int_{\gamma} x^{-s}  \frac{\Gamma(s) \Gamma(1-s)}{\Gamma(\beta - \alpha s)} \, ds.
\end{equation}
\noindent
The standard procedure gives 
\begin{equation}
f(x) = \sum_{n_{1},n_{2}} \phi_{12} \, \frac{1}{2 \pi i } \int_{\gamma}  
\left( \frac{\langle 1 - s + n_{2} \rangle x^{-s}}{\Gamma(\beta - \alpha s)} \right) \, \langle s + n_{1} \rangle \, ds.
\end{equation}
\noindent
Theorem \ref{thm-fun1} then produces 
\begin{equation}
f(x) = \sum_{n_{1},n_{2}} \phi_{12} \frac{x^{n_{1}}}{\Gamma(\beta+ \alpha n_{1})}
\langle 1 + n_{1}+n_{2} \rangle.
\end{equation}
\noindent
To evaluate this two-dimensional bracket series proceed as in Example \ref{example-2}. This gives 
\begin{equation}
\label{ml-1}
f(x) = \sum_{n=0}^{\infty} \frac{(-x)^{n} }{\Gamma(\beta + \alpha n)} \quad \text{when} \,\, |x| < 1
\end{equation}
\noindent
and 
\begin{equation}
\label{ml-2}
f(x) = \frac{1}{x} \sum_{n=0}^{\infty} \frac{1}{\Gamma(\beta - \alpha - \alpha n)} \frac{(-1)^{n}}{x^{n}} 
 \quad \text{when} \,\, |x| > 1.
 \end{equation}
 \noindent
 The function appearing in \eqref{ml-1} is the \texttt{Mittag-Leffler function}, defined in \cite{olver-2010a} by 
 \begin{equation}
 E_{\alpha,\beta}(z) = \sum_{n=0}^{\infty} \frac{z^{n}}{\Gamma(\alpha n + \beta)}.
 \label{ml-def}
 \end{equation}
 \noindent 
 produces the final expression 
 \begin{equation}
 f(x) = \begin{cases}
 E_{\alpha, \beta}(-x) & \quad \text{if} \,\, |x| < 1 \\
 x^{-1} E_{-\alpha, \beta-\alpha}(-1/x) & \quad \text{if} \,\, |x|>1.
 \end{cases}
 \end{equation}
\end{example}

\section{Direct computations of Mellin transforms}
\label{sec-direct}

This section describes how to use the method of brackets to produce the evaluation of the Mellin transform
\begin{equation}
\mathcal{M}(f(x))(s) = \int_{0}^{\infty} x^{s-1} f(x) \, dx.
\end{equation}

\begin{example}
\label{example-int1}
Example \ref{example-2} has produced the evaluation of 
\begin{equation}
\label{bes-int-2}
\int_{0}^{\infty}x^{\alpha-1} K_{\nu}(x) \, dx = 2^{\alpha-1} \Gamma \left( \frac{\alpha+\nu}{2} \right) 
\Gamma \left( \frac{\alpha-\nu}{2} \right),
\end{equation}
\noindent
from the Mellin inversion of $\Gamma(s-a)\Gamma(s-b)$.
\end{example}

\begin{example}
\label{example-int2}
The next example evaluates 
\begin{equation}
\label{int-mess1}
I(\alpha,\mu,\nu)  = \int_{0}^{\infty} x^{\alpha-1} K_{\mu}(x) K_{\nu}(x) \, dx.
\end{equation}
by the methods developed here. 

Entry $10.32.19$ in \cite{olver-2010a} contains the representation 
\begin{multline}
K_{\mu}(x)K_{\nu}(x) = \\ \frac{1}{8 \pi i } \int_{\gamma} 
\left( \frac{x}{2} \right)^{-2s} \frac{1}{\Gamma(2s)} 
\Gamma\left( s + \frac{\mu+\nu}{2} \right) 
\Gamma\left( s + \frac{\mu-\nu}{2} \right) 
\Gamma\left( s - \frac{\mu-\nu}{2} \right) 
\Gamma\left( s - \frac{\mu+\nu}{2} \right) 
\, ds
\end{multline}
\noindent
Replacing this in \eqref{int-mess1} and identifying the $x$-integral as a bracket, yields 
\begin{multline}
I(\alpha,\mu,\nu)  =  \frac{1}{8 \pi i } \int_{\gamma} 
 \frac{2^{2s} }{\Gamma(2s)} 
\Gamma\left( s + \frac{\mu+\nu}{2} \right) 
\Gamma\left( s + \frac{\mu-\nu}{2} \right) \\
\Gamma\left( s - \frac{\mu-\nu}{2} \right) 
\Gamma\left( s - \frac{\mu+\nu}{2} \right) 
  \langle \alpha - 2s \rangle
\, ds
\end{multline}
\noindent
Since this problem contains one bracket and one contour integral, there is no need to expand the gamma 
terms in bracket series and the result is obtained directly from Theorem \ref{thm-fun1}. The result is 
\begin{multline}
 \int_{0}^{\infty} x^{\alpha-1} K_{\mu}(x) K_{\nu}(x) \, dx  \\ =
 \frac{2^{\alpha-3}}{\Gamma(\alpha)} 
\Gamma \left( \frac{\alpha + \mu + \nu}{2} \right) 
\Gamma \left( \frac{\alpha + \mu - \nu}{2} \right) 
\Gamma \left( \frac{\alpha - \mu + \nu}{2} \right) 
\Gamma \left( \frac{\alpha - \mu - \nu}{2} \right).
\end{multline}
\end{example}

\begin{example}
\label{example-int3}
The evaluation of 
\begin{equation}
\int_{0}^{\infty} x^{2a-1} K_{\nu}^{2}(x) \, dx = 
 \frac{\sqrt{\pi}\, \Gamma(a+\nu) \Gamma(a-\nu) \Gamma(a)}{4\Gamma \left( a + \tfrac{1}{2} \right)} 
\end{equation}
\noindent
is the special case of Example \ref{example-int2} with $\mu = \nu$. Note that the parameter $a$ has been replaced
by $2a$, in order to write the answer in a more compact form. In particular, with $\nu=0$, this becomes 
\begin{equation}
\int_{0}^{\infty} x^{2a-1} K_{0}^{2}(x) \, dx = \frac{\sqrt{\pi} \Gamma^{3}(a)}{4 \, \Gamma \left(a + \tfrac{1}{2} \right)}.
\end{equation}
\noindent
The final special case mentioned here has $a = \tfrac{1}{2}$:
\begin{equation}
\int_{0}^{\infty} K_{0}^{2}(x) \, dx = \frac{\pi^{2}}{4}.
\end{equation}
\noindent
These examples have been evaluated in \cite{gonzalez-2017a} by a different procedure.
\end{example}

\begin{example}
Now consider the integral 
\begin{equation}
\varphi_{3}(a) = \int_{0}^{\infty} K_{0}^{3}(ax) \, dx
\end{equation}
\noindent
with an auxiliary parameter $a$ that naturally can be scaled out. 

The evaluation begins with a more general problem
\begin{equation}
I = I(a,b;\mu,\nu,\alpha)  = \int_{0}^{\infty} K_{\mu}(ax) K_{\nu}(ax) K_{\alpha}(bx) \, dx
\label{Kint-1}
\end{equation}
\noindent 
and then 
\begin{equation}
\label{Kint-2}
\varphi_{3}(a)  = \lim\limits_{\stackrel{\alpha = \mu = \nu \rightarrow 0}{b \rightarrow a}}  I(a,b;\mu,\nu,\alpha).
\end{equation}
\noindent
The Mellin-Barnes representations of the factors in the integrand 
\begin{multline}
\label{Kint-3}
K_{\mu}(ax) K_{\nu}(ax) = \\ \frac{1}{8 \pi i } \int_{\gamma} 
\Gamma \left( t + \frac{\mu + \nu}{2} \right) 
\Gamma \left( t + \frac{\mu - \nu}{2} \right) 
\Gamma \left( t - \frac{\mu + \nu}{2} \right) 
\Gamma \left( t - \frac{\mu - \nu}{2} \right) 
\left( \frac{ax}{2} \right)^{-2t} \frac{dt}{\Gamma(2t)},
\end{multline}
\noindent
and 
\begin{equation}
K_{\alpha}(bx) = \frac{1}{4 \pi i } \left( \frac{bx}{2} \right)^{\alpha} \int_{\gamma} \Gamma(s) \Gamma(s- \alpha) \left( \frac{bx}{2} \right)^{-2s} \, ds.
\label{Kint-4}
\end{equation}
\noindent
Replacing  in \eqref{Kint-1} gives 
\begin{multline}
\label{Kint-6}
I = \frac{1}{8(2 \pi i )^{2}} \int_{\gamma_{1}} \int_{\gamma_{2}} 
\frac{
\Gamma \left( t + \frac{\mu + \nu}{2} \right) 
\Gamma \left( t + \frac{\mu - \nu}{2} \right) 
\Gamma \left( t - \frac{\mu + \nu}{2} \right) 
\Gamma \left( t - \frac{\mu - \nu}{2} \right) 
\Gamma(s) \Gamma(s- \alpha)}
{a^{2t} b^{2s - \alpha} 2^{-2t - 2s + \alpha} \Gamma(2t)} \\
 \times \langle -2s - 2t + \alpha +1 \rangle \, dt \, ds.
\end{multline}
\noindent
Now replace the gamma factors in the denominator by their corresponding bracket series to obtain
\begin{eqnarray}
I & = & \frac{1}{8} \sum_{\{ n\}} \phi_{n_{1} \cdots n_{6}} 
\frac{1}{(2 \pi i)^{2}} 
\int_{\gamma_{1}} \int_{\gamma_{2}} 
\frac{1}{a^{2t} b^{2s - \alpha} 2^{-2t - 2s + \alpha} \Gamma(2t)} \label{Kint-7} \\
& \times & \langle t + \frac{\mu + \nu}{2} + n_{1} \rangle 
\langle t + \frac{\mu - \nu}{2} + n_{2} \rangle 
\langle t -  \frac{\mu + \nu}{2} + n_{3 } \rangle  \nonumber \\
& \times & \langle t - \frac{\mu - \nu}{2} + n_{4} \rangle 
\langle  s + n_{5} \rangle 
\langle s - \alpha + n_{6} \rangle \langle -2s - 2t + \alpha + 1 \rangle 
\, dt \, ds.   \nonumber 
\end{eqnarray}
\noindent
To evaluate this expression  choose to eliminate the sums with indices $n_{1}$ and $n_{5}$. The resulting 
sums now depend on four indices $n_{2}, \, n_{3}, \, n_{4}$ and $n_{6}$ and the variables of integration 
$t$ and $s$ take the values 
\begin{equation}
t^{*} = - \frac{\mu+\nu}{2} - n_{1} \quad \textnormal{and} \,\, s^{*} = - n_{5}.
\end{equation}
\noindent
This yields 
\begin{multline}
I = \frac{1}{8} \sum_{\{ n \}} \phi_{n_{2}n_{3}n_{4}n_{6}}  \times \\
\frac{ 
\langle t^{*} + \frac{\mu-\nu}{2} + n_{2} \rangle 
\langle t^{*} - \frac{\mu+\nu}{2} + n_{3}  \rangle 
\langle t^{*} - \frac{\mu - \nu}{2} + n_{4}  \rangle 
\langle s^{*} - \alpha + n_{6}   \rangle 
\langle -2s^{*} - 2t^{*} + \alpha + 1  \rangle 
}
{
a^{2t^{*}} b^{2s^{*} - \alpha} 2^{-2t^{*} - 2s^{*}+ \alpha} \Gamma(2t^{*})
}.
\end{multline}
\noindent
Under the assumption $|4a^{2}| < |b^{2}|$ the integral in \eqref{Kint-1} is expressed as 
\begin{equation}
I = I(a,b;\mu,\nu,\alpha)  = T_{1}+T_{2}+T_{3}+T_{4}
\label{expression-1}
\end{equation}
\noindent
with 
\begin{eqnarray}
\quad T_{1} & = & \frac{1}{8} \frac{a^{\mu+\nu}}{b^{1+ \mu + \nu}} \Gamma \left( \frac{1-\alpha + \mu + \nu}{2} \right) 
\Gamma \left( \frac{1 + \alpha + \mu + \nu}{2} \right) \Gamma(-\mu) \Gamma(-\nu) \label{Kint-11}  \\
& & \quad \quad \times \pFq43{1 + \tfrac{\mu+\nu}{2} \,\, \tfrac{1+ \mu + \nu}{2} \,\, \frac{1+ \alpha + \mu + \nu}{2} \,\, \frac{1 - \alpha + \mu + \nu}{2} }
{1+ \mu \,\, 1+ \nu \,\, 1+ \mu + \nu}{\frac{4a^{2}}{b^{2}}} \nonumber  \\
\quad T_{2} & = & \frac{1}{8} \frac{a^{\mu-\nu}}{b^{1+ \mu - \nu}} \Gamma \left( \frac{1-\alpha + \mu - \nu}{2} \right) 
\Gamma \left( \frac{1 + \alpha + \mu - \nu}{2} \right) \Gamma(-\mu) \Gamma(\nu) \nonumber  \\
& & \quad \quad \times \pFq43{1 + \tfrac{\mu-\nu}{2} \,\, \tfrac{1+ \mu - \nu}{2} \,\, \frac{1- \alpha + \mu - \nu}{2} \,\, \frac{1 + \alpha + \mu - \nu}{2} }
{1+ \mu \,\, 1- \nu \,\, 1+ \mu - \nu}{\frac{4a^{2}}{b^{2}}} \nonumber   \\
\quad T_{3} & = & \frac{1}{8} \frac{b^{\mu+\nu-1}}{a^{\mu + \nu}} \Gamma \left( \frac{1-\alpha - \mu - \nu}{2} \right) 
\Gamma \left( \frac{1 + \alpha - \mu - \nu}{2} \right) \Gamma(\mu) \Gamma(\nu) \nonumber  \\
& & \quad \quad \times \pFq43{1 - \tfrac{\mu+\nu}{2} \,\, \tfrac{1- \mu - \nu}{2} \,\, \frac{1-\alpha - \mu - \nu}{2} \,\, \frac{1 + \alpha - \mu - \nu}{2} }
{1- \mu \,\, 1- \nu \,\, 1- \mu - \nu}{\frac{4a^{2}}{b^{2}}} \nonumber   \\
\quad T_{4} & = & \frac{1}{8} \frac{b^{\mu-\nu-1}}{a^{\mu - \nu}} \Gamma \left( \frac{1-\alpha - \mu + \nu}{2} \right) 
\Gamma \left( \frac{1 + \alpha - \mu + \nu}{2} \right) \Gamma(\mu) \Gamma(-\nu) \nonumber  \\
& & \quad \quad \times \pFq43{1 - \tfrac{\mu-\nu}{2} \,\, \tfrac{1- \mu + \nu}{2} \,\, \frac{1-\alpha - \mu + \nu}{2} \,\, \frac{1 + \alpha - \mu + \nu}{2} }
{1- \mu \,\, 1+ \nu \,\, 1- \mu +  \nu}{\frac{4a^{2}}{b^{2}}}. \nonumber 
\end{eqnarray}
\noindent
Now passing to the limit as $\alpha, \, \nu, \, \mu  \rightarrow 0$ yields 
\begin{eqnarray}
\int_{0}^{\infty} K_{0}^{2}(ax) K_{0}(bx) \, dx & = & 
\lim\limits_{\mu \rightarrow 0} \frac{1}{8b} 
\left[ \left( \frac{a^{2}}{b^{2}} \right)^{\mu} \Gamma \left( \frac{1 + 2 \mu}{2} \right)^{2} \Gamma(-\mu)^{2} 
\pFq32{ \tfrac{1+2 \mu}{2} \,\, \tfrac{1+ 2 \mu}{2} \,\, \tfrac{1+ 2 \mu}{2}  }
{1+ \mu \,\, 1+ 2 \mu}{\frac{4a^{2}}{b^{2}}} \nonumber  \right.  \\
& & \quad \quad + 2 \pi \Gamma(-\mu) \Gamma(\mu) \pFq32{ \tfrac{1}{2} \,\,\, \tfrac{1}{2} \,\,\, \tfrac{1}{2} }{1+\mu \,\, 1 - \mu}  { \frac{4a^{2}}{b^{2}}}   \nonumber \\
& & \left. \quad \quad + \left( \frac{a^{2}}{b^{2}} \right)^{-\mu} \Gamma^{2} \left(  \tfrac{1- 2 \mu}{2} \right) \Gamma^{2}(\mu) 
\pFq32{ \tfrac{1- 2 \mu}{2} \,\, \tfrac{1- 2 \mu}{2} \,\, \tfrac{1 - 2 \mu}{2} }{1- \mu \,\,\, 1 - 2 \mu} { \frac{4a^{2}}{b^{2}}} 
\right] \nonumber 
\end{eqnarray}

In a similar form, in the case $|4a^{2}| >|b^{2}|$ one obtains 
\begin{equation}
I = I(a,b;\mu,\nu,\alpha)  = T_{5}+T_{6}
\label{expression-2}
\end{equation}
\noindent
with 
\begin{eqnarray}
\quad T_{5} & = & \frac{1}{8} \frac{b^{\alpha}}{a^{\alpha +1}}  \nonumber \\
& & \frac{\Gamma(- \alpha)}{\Gamma(\alpha+1)} 
\Gamma \left( \frac{1+\alpha - \mu + \nu}{2} \right) 
\Gamma \left( \frac{1 + \alpha - \mu -  \nu}{2} \right) 
\Gamma \left( \frac{1 + \alpha + \mu -  \nu}{2} \right) 
\Gamma \left( \frac{1 + \alpha + \mu +  \nu}{2} \right)  \nonumber \\
& & \quad \quad \times \pFq43{\tfrac{1+\alpha+ \mu+\nu}{2} \,\, \tfrac{1+\alpha - \mu + \nu}{2} \,\, \frac{1+ \alpha + \mu - \nu}{2} \,\, \frac{1 + \alpha - \mu - \nu}{2} }
{1+ \alpha  \,\, 1+ \tfrac{\alpha}{2}  \,\, \tfrac{1+ \alpha}{2} }{\frac{b^{2}}{4a^{2}}} \nonumber  \\
\quad T_{6} & = & \frac{1}{8} \frac{a^{\alpha-1}}{b^{\alpha }}  \nonumber \\
& & \frac{\Gamma(\alpha)}{\Gamma(1-\alpha)} 
\Gamma \left( \frac{1-\alpha - \mu + \nu}{2} \right) 
\Gamma \left( \frac{1 - \alpha - \mu - \nu}{2} \right) 
\Gamma \left( \frac{1 - \alpha + \mu -  \nu}{2} \right) 
\Gamma \left( \frac{1 - \alpha + \mu +  \nu}{2} \right)  \nonumber \\
& & \quad \quad \times \pFq43{\tfrac{1-\alpha+ \mu+\nu}{2} \,\, \tfrac{1-\alpha - \mu + \nu}{2} \,\, \frac{1- \alpha + \mu - \nu}{2} \,\, \frac{1 - \alpha - \mu - \nu}{2} }
{1- \alpha  \,\, 1- \tfrac{\alpha}{2}  \,\, \tfrac{1- \alpha}{2} }{\frac{b^{2}}{4a^{2}}} \nonumber 
\end{eqnarray}
\noindent
and then letting $\mu = \nu = 0, \, b=a$ and $\alpha \rightarrow 0$ yields, after scaling the parameter $a$
\begin{multline}
\label{Kint-20} 
\int_{0}^{\infty} K_{0}^{3}(x)   = \lim\limits_{\alpha \rightarrow 0} 
\frac{1}{8} \left[ \frac{\Gamma(- \alpha) \Gamma^{4} \left( \frac{1+ \alpha}{2} \right) }{\Gamma(1 + \alpha)} 
\pFq32{ \tfrac{1+ \alpha}{2} \,\, \tfrac{1+ \alpha}{2} \,\, \tfrac{1+\alpha}{2} }{ 1+ \alpha \,\, 1+ \tfrac{\alpha}{2} }{\frac{1}{4} } \right.  \\
\left. + \frac{\Gamma(\alpha) \Gamma^{4} \left( \tfrac{1 - \alpha}{2} \right)}{\Gamma(1 - \alpha)} 
\pFq32{ \tfrac{1- \alpha}{2} \,\, \tfrac{1- \alpha}{2} \,\, \tfrac{1-\alpha}{2} }{ 1 - \alpha \,\, 1- \tfrac{\alpha}{2} }{\frac{1}{4} }  \right].
\end{multline}

The authors have been unable to produce a simpler analytic expression for this limiting value. 
\end{example}

\begin{example}
A similar argument as the one presented in the previous example yields 
\begin{multline}
\label{Kint-41} 
\int_{0}^{\infty} K_{0}^{4}(x) \, dx    = \frac{\sqrt{\pi}}{16} \lim\limits_{\alpha \rightarrow 0}  \left[
\frac{\Gamma^{2}(-\alpha) \Gamma^{2} \left( \alpha + \tfrac{1}{2} \right) \Gamma \left( 2 \alpha + \tfrac{1}{2} \right)}{\Gamma(2 \alpha + 1)}
\pFq43{ \tfrac{1}{2} \,\, \alpha + \tfrac{1}{2} \,\, \alpha + \tfrac{1}{2} \,\, 2 \alpha + \tfrac{1}{2}}{1+ \alpha \,\, 1+ \alpha \,\, 1+ 2 \alpha }{1}  \right. \\
\left. + 2 \sqrt{\pi} \Gamma(-\alpha) \Gamma(\alpha)  \Gamma( \tfrac{1}{2} + \alpha ) \Gamma( \tfrac{1}{2} - \alpha) 
\pFq43{ \tfrac{1}{2} \,\,\,\, \tfrac{1}{2} \,\,\,\,  \tfrac{1}{2} + \alpha \,\,\,\,  \tfrac{1}{2} - \alpha }{1 \,\,\,\, 1+ \alpha \,\,\,\, 1-  \alpha }{1}  \right. \\
\left.  + \frac{\Gamma^{2}(\alpha) \Gamma^{2}( \tfrac{1}{2} - \alpha) \Gamma( \tfrac{1}{2} - 2 \alpha)}{\Gamma(1 -2 \alpha)} 
\pFq43{ \tfrac{1}{2} \,\,\,\, \tfrac{1}{2} - \alpha  \,\,\,\, \tfrac{1}{2} - \alpha \,\,\,\,  \tfrac{1}{2} - 2 \alpha }{1-  \alpha \,\,\,\, 1 -  \alpha \,\,\,\, 1- 2  \alpha }{1}  \right]
\end{multline}
\noindent
Details are omitted.
\end{example}

\section{Mellin transforms of products}
\label{sec-special}

This section presents a method to evaluate the Mellin transform 
\begin{equation}
I(s,b) = \int_{0}^{\infty} x^{s-1} f(x) g(bx) \,dx.
\end{equation}
\noindent
knowing a series for $f$ of the form
\begin{equation}
f(x) = \sum_{n} \phi_{n} F(n)x^{\beta n}
\end{equation}
\noindent
and the inverse Mellin transform 
\begin{equation}
g(x) = \frac{1}{2 \pi i } \int_{\gamma}x^{-s}  \varphi(s)  \, ds.
\end{equation}
\noindent
For $\kappa > 0$, the change of variables $s = \kappa s'$ gives 
\begin{equation}
g(x) = \frac{1}{2 \pi i } \int_{- i \infty}^{i \infty}x^{-\kappa s}  \widetilde{\varphi}(s)  \, ds,
\label{rep-gamma}
\end{equation}
\noindent
where $\widetilde{\varphi}(s) = \kappa \varphi(\kappa s)$.  The formula \eqref{rep-gamma} is now written as
\begin{equation}
g(x) = \frac{1}{2 \pi i } \int_{\gamma}x^{-\kappa s}  \varphi(s)  \, ds,
\label{rep-gamma1}
\end{equation}
\noindent
that is, the tilde notation is dropped and the parameter $\kappa$ is kept.

Then 
\begin{eqnarray}
I & = & \int_{0}^{\infty} x^{\alpha -1} f(x) g(bx) \, dx \\
& = & \frac{1}{2 \pi i } \int_{\gamma} \varphi(s) b^{- \kappa s} 
\left[ \sum_{n} \phi_{n} F(n) \left[ \int_{0}^{\infty} x^{\alpha + \beta n - \kappa s -1} \, dx \right]  \right] \, ds 
\nonumber \\
& = & \frac{1}{2 \pi i } \int_{\gamma} \varphi(s) b^{- \kappa s} 
\left[ \sum_{n} \phi_{n} F(n) \langle \alpha + \beta n - \kappa s \rangle   \right] \, ds 
\nonumber  \\
& = & \frac{1}{2 \pi i }   \int_{\gamma} \frac{\varphi(s) b^{- \kappa s} }{| \beta |}
\left[ \sum_{n}  \phi_{n} F(n)  \left \langle \frac{\alpha}{\beta}   - \frac{\kappa s}{\beta} + n  \right \rangle   \right] \, ds 
\nonumber   \\
& = & \frac{1}{ | \beta | } \frac{1}{2 \pi i } 
\int_{\gamma} \varphi(s) b^{-\kappa s} \Gamma \left( \frac{\alpha - \kappa s}{\beta} \right) 
F \left( \frac{\kappa s - \alpha}{\beta} \right)  \, ds. \nonumber 
\end{eqnarray}

This is stated as a theorem.

\begin{theorem}
\label{thm-identity1}
Assume the function $f(x)$ has an expansion given by 
\begin{equation}
\label{exp-f}
f(x) = \sum_{n} \phi_{n} F(n) x^{\beta n}
\end{equation}
\noindent
and the function $g(x)$ is given by rescaled version of the inverse Mellin transform 
\begin{equation}
\label{exp-g}
g(x) = \frac{1}{2 \pi i } \int_{\gamma} \varphi(s) x^{-\kappa  s} \, ds.
\end{equation}
\noindent
Then 
\begin{equation*}
\int_{0}^{\infty} x^{\alpha - 1} f(x) g(bx) \, dx = 
 \frac{1}{ | \beta | } \frac{1}{2 \pi i } 
\int_{\gamma} \varphi(s) \Gamma \left( \frac{\alpha - \kappa s}{\beta} \right) 
F \left( \frac{\kappa s - \alpha}{\beta} \right) b^{- \kappa s} \, ds. 
\end{equation*}
\end{theorem}

\begin{example}
Entry $6.532.4$ in \cite{gradshteyn-2015a} states that 
\begin{equation}
\label{65324} 
\int_{0}^{\infty} \frac{x \, J_{0}(Ax) \, dx}{x^{2}+k^{2}} = K_{0}(Ak).
\end{equation}
\noindent
Theorem \ref{thm-identity1} is now used to establish this evaluation.

Start with the expansion 
\begin{equation}
f(x) = \frac{1}{x^{2}+k^{2}}  = \sum_{n} \phi_{n} \left[ \Gamma(n+1) k^{-2n-2} \right] x^{2n}.
\end{equation}
\noindent
This gives \eqref{exp-f} with $F(n) = \Gamma(n+1)k^{-2n-2}$ and $\beta = 2$.

Now use \cite[10.9.23]{olver-2010a} 
\begin{equation}
J_{\nu}(z) = \frac{1}{2 \pi i } \int_{\gamma} \frac{\Gamma(t) }{\Gamma(\nu-t+1)} \left( \frac{z}{2} \right)^{\nu - 2t} \, dt,
\end{equation}
\noindent
%to write $J_{0}(x)$ as an inverse Mellin transform. 
%Equation \eqref{bessel-mellin1} is 
%\begin{equation}
%\frac{1}{2 \pi i } \int_{\gamma} x^{-s} \frac{\Gamma(a-s)}{\Gamma(s-b)} \, ds = 
%x^{-\tfrac{1}{2}(a+b+1)} J_{a-b-1} \left( \frac{2}{\sqrt{x}} \right).
%\end{equation}
%\noindent
%In particular, with $a=0$ and $b=-1$ it gives 
%\begin{equation}
%J_{0} \left( \frac{2}{\sqrt{x}} \right) = \frac{1}{2 \pi i } \int_{\gamma} x^{-s} \frac{\Gamma(-s)}{\Gamma(s+1)} \,  ds
%\end{equation}
%\noindent
and  replace $z$ by $Ax$ to produce the desired representation of the Bessel function:
\begin{equation}
J_{0}(Ax) = \frac{1}{2 \pi i } \int_{\gamma}
 \frac{(Ax)^{2s} \Gamma(s) }{2^{2s} \Gamma(1-s)}  \, ds.
\end{equation}
\noindent
In the notation of \eqref{exp-g}, the parameters are  $\kappa=2, \, \beta = 2, \, \alpha = 2$ and 
$b=A$  and 
$\begin{displaystyle}
\varphi(s) = \frac{\Gamma(s)}{2^{2s} \, \Gamma(1-s)}
\end{displaystyle}$. Theorem \ref{thm-identity1} now gives 
\begin{equation}
\label{bessel-last1}
\int_{0}^{\infty} \frac{x J_{0}(Ax) }{x^{2}+k^{2}} \, dx =
\frac{1}{4 \pi i } \int_{ \gamma} \Gamma^{2}(s) \left( \frac{Ak}{2} \right)^{-2s} \, ds.
\end{equation}
Now the formula \cite[10.32.13]{olver-2010a})
\begin{equation}
\label{mellin-bes2}
K_{\nu}(z) = \frac{1}{4 \pi i } \left( \frac{z}{2} \right)^{\nu} 
\int_{\gamma} \Gamma(t) \Gamma(t- \nu) \left( \frac{z}{2} \right)^{-2t} \, dt,
\end{equation}
\noindent
In particular for $\nu = 0$ this becomes 
\begin{equation}
\label{bessel-last2}
K_{0}(z)  = \frac{1}{4 \pi i } \int_{\gamma} \Gamma^{2}(t) \left( \frac{z}{2} \right)^{-2t} \, dt.
\end{equation}
\noindent
Formula \eqref{65324} now follows from \eqref{bessel-last1} and \eqref{bessel-last2}.
\end{example}

\begin{example}
The next evaluation is 
\begin{equation}
\label{65212}
I = \int_{0}^{\infty} x K_{\nu}(ax) J_{\nu}(bx) \, dx = \frac{b^{\nu}}{a^{\nu}(a^{2}+b^{2})}.
\end{equation}
\noindent
This is entry $6.521.2$ in \cite{gradshteyn-2015a}. 

Formulas \eqref{mellin-bes1} and \eqref{mellin-bes3} give, respectively,  the representations 
\begin{equation}
K_{\nu}(z) = \frac{1}{4 \pi i } \int_{\gamma} \Gamma(s) \Gamma(s-\nu) \,
\left( \frac{z}{2} \right)^{-2s+\nu}  \, ds
\end{equation}
\noindent
and 
\begin{equation}
J_{\nu}(z) = \frac{1}{2 \pi  i} \int_{\gamma} \frac{\Gamma(-s)}{\Gamma(s + \nu +1)} 
\left( \frac{z}{2} \right)^{2s+\nu} \, ds.
\end{equation}
\noindent
Replacing in \eqref{65212} and recognizing the $x$-integral into a bracket yields 
\begin{multline}
I = \frac{1}{2 (2 \pi  i)^{2}} \left( \frac{ab}{4} \right)^{\nu} 
\int_{\gamma_{1}} \int_{\gamma_{2}} 
\frac{\Gamma(-t) \Gamma(s) \Gamma(s- \nu) }{\Gamma(t + \nu + 1) } 
\left( \frac{a}{2} \right)^{- 2s} \left( \frac{b}{2} \right)^{2t} \\
\times  \frac{1}{2} \langle 1 + \nu + t  - s \rangle \, ds \, dt.
\end{multline}
\noindent
The bracket is now used to eliminate the $s$-integral, the result is 
\begin{equation}
I = \frac{1}{a^{2}} \left( \frac{b}{a} \right)^{\nu} 
\frac{1}{2 \pi i } \int_{\gamma} \Gamma(-t) \Gamma(1+t) \left( \frac{b^{2}}{a^{2}} \right)^{t} \, dt.
\end{equation}

The gamma terms are expanded as bracket series using \eqref{gamma-brack1} and then 
eliminate the line integral with one of the brackets  to obtain 
\begin{eqnarray}
I & = &  \frac{b^{ \nu}}{a^{ \nu +2}}  \frac{1}{2 \pi i } \int_{\gamma} 
 \sum_{n_{1}, n_{2}} \left( \frac{b^{2}}{a^{2}} \right)^{t} \phi_{n_{1}n_{2}}  \langle n_{1} - t \rangle 
 \langle n_{2} +  1 + t \rangle \, dt \\
 & = & \frac{b^{\nu}}{a^{\nu+2}} \sum_{n_{1},n_{2}} \phi_{n_{1}n_{2}} \left( \frac{b^{2}}{a^{2}} \right)^{n_{1}} 
 \langle n_{2}+n_{1}+1 \rangle. \nonumber
 \end{eqnarray}
\noindent
The method of brackets now gives two different expressions obtained by using $n_{1}$ or $n_{2}$ as the 
free index of summation. These options give series converging in disjoint regions $|b|< |a|$ and 
$|b|> |a|$. The rules of the method of brackets shows that  each sum gives the value of the integral in the 
corresponding region. In this case, both cases  give the same expression 
\begin{equation}
I = \frac{b^{\nu}}{a^{\nu}(a^{2}+b^{2})},
\end{equation}
\noindent
for the value of the integral. This confirms \eqref{65212}.
\end{example}

\section{An integral involving exponentials and Bessel modified functions}
\label{sec-exp-bessel}

The method developed in this work is now applied to the computation of some definite integrals. The idea is 
relatively simple: given a function $f(x)$ with a  Mellin transform $\varphi(s)$  containing gamma factors, the integral 
\begin{equation}
F = \int_{0}^{\infty} e^{-x} f(x) \, dx
\end{equation}
\noindent
can be evaluated by writing the exponential as 
\begin{equation}
e^{-x} = \sum_{n} \phi_{n}x^{n}
\label{exp-ser1}
\end{equation}
\noindent 
and then using the method of brackets. The examples below illustrates this process. 

Integrals involving the Bessel function $K_{\nu}(x)$ have been presented in \cite{gonzalez-2017a}. The computation 
there is based on the concept of \textit{totally null/divergent representations}. The first type includes
\begin{equation}
K_{0}(x) = \frac{1}{x} \sum_{n=0}^{\infty} \frac{\Gamma^{2}(n + \tfrac{1}{2})}{n! \, \Gamma(-n)} 
\left( - \frac{4}{x^{2}} \right)^{n}
\end{equation}
\noindent
in which every term vanishes. There is a similar expression for a series in which every term diverges. In spite of 
the lack of rigor, these series have shown to be useful in the evaluation of definite integrals. See 
\cite{gonzalez-2017a} for details.  A second technique used before is the integral 
representation of $K_{\nu}(x)$ such as 
\begin{equation}
K_{\nu}(x) = \frac{2^{\nu} \Gamma \left( \nu + \tfrac{1}{2} \right)}{\Gamma(\tfrac{1}{2})} x^{\nu} 
\int_{0}^{\infty} \frac{\cos t  \, dt}{(x^{2}+t^{2})^{\nu+1/2}},
\end{equation}
\noindent
and then apply the methods of brackets.  The method present next is an improvement over those employed in \cite{gonzalez-2017a}. 

\begin{example}
Consider the integral 
\begin{equation}
F(b,\nu) = \int_{0}^{\infty} e^{-x} K_{\nu}(bx) \, dx.
\end{equation}
\noindent
Naturally the parameter $b$ can be scaled out, but it is instructive to leave it as is. 

Write the exponential function as in \eqref{exp-ser1} and  the Bessel function from \eqref{mellin-bes1} as 
\begin{equation}
\label{mellin-bes2a}
K_{\nu}(bx) = \frac{1}{4 \pi i } \left( \frac{bx}{2} \right)^{\nu} 
\int_{-i \infty}^{i \infty} \Gamma(s) \Gamma(s- \nu) \left( \frac{bx}{2} \right)^{-2s} \, ds.
\end{equation}
\noindent
Then
\begin{eqnarray}
F(b,\nu) & = &  \frac{1}{2} \sum_{n_{1}} \phi_{n_{1}} \frac{1}{2 \pi i } 
\int_{\gamma} \left( \frac{b}{2} \right)^{\nu - 2s} \Gamma(s) \Gamma(s- \nu) 
\left( \int_{0}^{\infty} x^{n_{1}-2s + \nu} \, dx \right) \, ds \nonumber \\
& = &  \frac{1}{2} \sum_{n_{1}} \phi_{n_{1}} \frac{1}{2 \pi i } 
\int_{\gamma} \left( \frac{b}{2} \right)^{\nu - 2s} \Gamma(s) \Gamma(s- \nu) 
\langle n_{1} - 2 s + \nu +1 \rangle \, ds \nonumber 
\end{eqnarray}
\noindent
Now replace the gamma factors by their bracket expansions as in \eqref{gamma-brack1} to produce 
\begin{equation}
F(b,\nu) = 
\frac{1}{2} \sum_{n_{1}} \phi_{123} \frac{1}{2 \pi i } 
\int_{\gamma} \left( \frac{b}{2} \right)^{\nu - 2s} 
\langle s + n_{2} \rangle \, \langle s - \nu + n_{3} \rangle \, \langle n_{1} - 2s + \nu + 1 \rangle 
 \, ds \nonumber 
 \end{equation}
 \noindent
 Using the bracket $\langle s + n_{2} \rangle$ to apply Theorem \ref{thm-fun1} gives a bracket series for the 
 integral $F(b,\nu)$:
 \begin{equation}
 \label{value-f1}
 F(b,\nu) = \frac{1}{2} \sum_{\vec{n}} \phi_{123} 
 \left( \frac{b}{2} \right)^{\nu+2n_{2}} \langle -n_{2} -\nu + n_{3} \rangle 
 \langle n_{1} + 2 n_{2} + \nu + 1 \rangle,
 \end{equation}
 \noindent
 where $\vec{n} = \{n_{1},n_{2},n_{3}\}$.
 
 The method of brackets is now used to produce three sums for \eqref{value-f1}: 
 \begin{eqnarray}
 T_{1} & = & \frac{1}{4} \sum_{n=0}^{\infty} \frac{(-1)^{n}}{n!} 
 \Gamma \left( \frac{1 - \nu +n}{2} \right) 
 \Gamma \left( \frac{1+ \nu + n}{2} \right) 
 \left( \frac{b}{2} \right)^{-n-1} \label{tsum-1} \\
 T_{2} & = & \frac{1}{2} \sum_{n=0}^{\infty} \frac{(-1)^{n}}{n!} \Gamma(-n-\nu) \Gamma(1+2n+\nu) 
 \left( \frac{b}{2} \right)^{\nu + 2n} \label{tsum-2} \\
 T_{3} & = & \frac{1}{2} \sum_{n=0}^{\infty} 
 \frac{(-1)^{n}}{n!} \Gamma(\nu-n) \Gamma(1+2n - \nu) \left( \frac{b}{2} \right)^{-\nu+2n}
 \label{tsum-3}
 \end{eqnarray}
 \noindent
 Observe that $T_{1}$ diverges since it is of type ${_{2}}F_{0}$ (this sum plays a role in the asymptotic study 
 of the integral when $b \rightarrow \infty$, not described here) and the series  $T_{2}, \, T_{3}$ converge when $|b|<1$. The method 
 of brackets now states that $F(b,\nu) = T_{2}+T_{3}$, when $|b| <1$. Under this assumption 
 and since $T_{3}(b,\nu) = T_{2}(b, - \nu)$, it suffices to obtain 
 an expression for $T_{2}(b, \nu)$.  This is the next result.

 \begin{proposition}
 The function $T_{2}(b, \nu)$ is given by
 \begin{equation}
 T_{2}(b,\nu)  = - \frac{\pi}{2 \sin(\pi \nu)} \frac{b^{\nu} (1 + \sqrt{1 - b^{2}})^{- \nu}}{\sqrt{1-b^{2}}}. 
 \end{equation}
 \end{proposition}
 \begin{proof}
 The proof is divided in a sequence of steps.
 
 \noindent
 \texttt{Step 1}. The function $T_{2}(b,\nu)$ is given by 
 \begin{equation}
 T_{2}(b, \nu) = - \frac{\pi b^{\nu}}{2^{\nu+1} \sin(\pi \nu)} 
 \sum_{n=0}^{\infty} \frac{\left( \frac{1+\nu}{2} \right)_{n} \, \left( 1 + \frac{\nu}{2} \right)_{n}}{(1+\nu)_{n}} 
 \, \frac{b^{2n}}{n!}.
 \end{equation}
 \noindent
 \textit{Proof}.  In the definition of $T_{2}(b,\nu)$ use 
 $(w)_{2n} = 2^{2n} \left( \frac{w}{2} \right)_{n} \, \left(  \frac{w+1}{2} \right)_{n},$
 \begin{equation*}
 \Gamma(-n-\nu) = \frac{(-1)^{n} \Gamma(-\nu)}{(1+\nu)_{n}}, \,\,
 \Gamma(1+\nu + 2n) = \Gamma(1+\nu) 2^{2n} \left( \frac{1+\nu}{2} \right)_{n} \left(1 + \frac{\nu}{2} \right)_{n} 
 \end{equation*}
 \noindent
 and $\Gamma(-\nu) \Gamma(1+\nu)  = - \pi/\sin(\pi \nu)$, to obtain the result. 
 
 This identity shows that the statement of the Proposition is equivalent to proving
 \begin{equation}
 \frac{1}{2^{\nu}} \sum_{n=0}^{\infty} 
 \frac{\left( \frac{1+\nu}{2} \right)_{n} \, \left( 1 + \frac{\nu}{2} \right)_{n}}{(1+\nu)_{n}} \frac{b^{2n}}{n!} =
 \frac{(1+ \sqrt{1-b^{2}} \,\, )^{-\nu}}{\sqrt{1-b^{2}}}.
 \end{equation}
 
 \smallskip
 
 \noindent
 \texttt{Step 2}. Formula (2.5.16) in H.~Wilf \cite{wilf-1990a} 
 \begin{equation}
 \left( \frac{1 - \sqrt{1-4x} }{2x} \right)^{k} = \sum_{n=0}^{\infty} \frac{k \, (2n+k-1)!}{n! \, (n+k)!} x^{n}
 \end{equation}
 \noindent
 and the identity $(1+ \sqrt{1-b^{2}})(1 - \sqrt{1- b^{2}}) = b^{2}$ shows, after some elementary simplifications, 
  that the statement of the Proposition is equivalent to 
 \begin{equation}
 \sum_{n=0}^{\infty} 
  \frac{\left( \frac{1+\nu}{2} \right)_{n} \, \left( 1 + \frac{\nu}{2} \right)_{n}}{(1+\nu)_{n}} \frac{c^{n}}{n!} =
  \frac{1}{\sqrt{1-c}}
  \sum_{n=0}^{\infty} 
   \frac{\left( \frac{\nu}{2} \right)_{n} \, \left( 1 + \frac{\nu}{2} \right)_{n}}{(1+\nu)_{n}} \frac{c^{n}}{n!},
   \end{equation}
   \noindent
   where $c = b^{2}$.
   
   \smallskip
   
   \noindent
   \texttt{Step 3}. The identity at the end of Step 2 is equivalent to 
   \begin{equation}
   \pFq21{\frac{1+\nu}{2} \,\, 1 + \frac{\nu}{2}}{1+ \nu}{u} =
   \pFq10{\tfrac{1}{2}}{-}{u} \times 
    \pFq21{\frac{1+\nu}{2} \,\,  \frac{\nu}{2}}{1+ \nu}{u}.
   \end{equation}
   \noindent
   \textit{Proof}. Simply use the binomial theorem 
   \begin{equation}
   (1 - t)^{-\alpha} = \sum_{n=0}^{\infty} \frac{(\alpha)_{n}}{n!} t^{n}.
   \end{equation}
   \noindent
   with $\alpha = 1/2$.
   
   \smallskip
   
   \noindent
   \texttt{Step 4}. The conclusion of the Proposition is equivalent to the identity
   \begin{equation}
   \label{new-iden1}
   \sum_{j=0}^{n} \frac{\binom{n}{j} }{(1+ \nu)_{n-j}} 
   \left( \frac{1}{2} \right)_{j} \left( \frac{1+ \nu}{2} \right)_{n-j} \left( \frac{\nu}{2} \right)_{n-j} =
    \frac{\left( \frac{\nu}{2} \right)_{n} \, \left( 1 + \frac{\nu}{2} \right)_{n}}{(1+\nu)_{n}},
    \end{equation}
    \noindent
    for every $n \in \mathbb{N}$.
    
    \end{proof}
    
    \smallskip
    
    \begin{proof}
    The umbral method \cite{gessel-2003a,roman-1978a} shows that the identity is equivalent to 
    the one formed by replacing Pochhammer symbols $(a)_{k}$ by $a^{k}$.  In this case, \eqref{new-iden1} 
    becomes 
    \begin{equation}
    \sum_{j=0}^{n} \binom{n}{j} 2^{j} \nu^{n-j} = (\nu+2)^{n},
    \end{equation}
    \noindent
    and this  follows from the binomial theorem. 
    \end{proof}
    \begin{proof}
An alternative proof of \eqref{new-iden1} is presented next.  Expressing the Pochhammer symbols in terms of 
binomial coefficients, it is routine to check that the desired identity is equivalent to 
\begin{equation}
\sum_{j=0}^{n} \binom{2j}{j} \binom{\nu+2(n-j)}{n-j} \frac{\nu}{\nu+2(n-j)} = \binom{\nu+2n}{n}.
\label{new-iden2}
\end{equation}
\noindent
This identity is interpreted as the coefficient of $x^{n}$ in the product of the two series 
\begin{equation}
A(x) = \sum_{j=0}^{\infty} \binom{2j}{j}x^{j} \quad \text{and} \quad 
B(x) = \sum_{k=0}^{\infty} \nu \binom{\nu+2k}{k} \frac{x^{k}}{\nu+2k}.
\end{equation}
\noindent
The first sum is given by the binomial theorem as 
\begin{equation}
A(x) = \frac{1}{\sqrt{1-4x}}.
\end{equation}
\noindent
To obtain an analytic expression for $B(x)$ start with entry $2.5.15$ in \cite{wilf-1990a}
\begin{equation}
\frac{1}{\sqrt{1-4x}} \left( \frac{1 - \sqrt{1-4x}}{2x} \right)^{\nu} = 
\sum_{n=0}^{\infty} \binom{\nu+2k}{k} x^{k}
\end{equation}
\noindent
where the term in brackets is the generating function of the Catalan numbers.  Some elementary manipulations 
give 
\begin{equation}
B(x) = \nu x^{-\nu/2} \int_{0}^{\sqrt{x}} \frac{t^{\nu-1}}{\sqrt{1-4t^{2}}} 
\left( \frac{1 - \sqrt{1-4t^{2}}}{2t^{2}} \right)^{\nu} \, dt.
\end{equation}
\noindent
Then \eqref{new-iden2} is equivalent to the identity
\begin{equation}
\int_{0}^{\sqrt{x}} \frac{t^{\nu-1}}{\sqrt{1-4t^{2}}} \left( \frac{1 - \sqrt{1-4t^{2}}}{2t^{2}} \right)^{\nu} \, dt = 
\frac{1}{\nu} x^{\nu/2} \left( \frac{1-\sqrt{1- 4x}}{2x} \right)^{\nu}.
\end{equation}
\noindent
This can now be verified by observing that both sides vanish at $x=0$ and a direct computation shows that the 
derivatives match.
\end{proof}

 The integral is stated next. It  appears as entry $6.611.3$ in \cite{gradshteyn-2015a}. 
  
  \begin{corollary}
  For $\nu, \, b \in \mathbb{R}$ 
  \begin{equation}
  \int_{0}^{\infty} e^{-x} K_{\nu}(bx) \, dx = \\
  \frac{\pi}{2 b^{\nu} \sin(\pi \nu)}
  \left[ \frac{(1+ \sqrt{1-b^{2}})^{\nu} - (1 - \sqrt{1 - b^{2}})^{\nu}}{\sqrt{1-b^{2}}} \right].
  \end{equation}
  \noindent
  In particular, as $b \rightarrow 1$,
  \begin{equation}
  \int_{0}^{\infty} e^{-x} K_{\nu}(x) \, dx = \\
  \frac{\pi \nu}{ \sin(\pi \nu)},
    \end{equation}
and as $\nu \rightarrow 0$, 
\begin{equation}
\int_{0}^{\infty} e^{-x}K_{0}(bx) \, dx =   \frac{\log(1+\sqrt{1-b^{2}})  - \log(1-\sqrt{1-b^2})}{2 \sqrt{1-b^2}}.
\end{equation}
\noindent
  Finally, letting $\nu \rightarrow 0$  and $b \rightarrow 1$ gives 
    \begin{equation}
  \int_{0}^{\infty} e^{-x} K_{0}(x) \, dx = 1.
    \end{equation}
  \end{corollary}
\end{example}

\section{Flexibility of the method of brackets}
\label{sec-flexibility}

This final section illustrates the flexibility of the method of brackets. To achieve this goal, we present four 
different evaluations of the integral 
\begin{equation}
\label{flex-0}
I(a,b) = \int_{0}^{\infty} \frac{e^{-a^{2}x^{2}} \, dx}{x^{2} + b^{2}}.
\end{equation}
\noindent
The parameters $a,\, b$ are assumed to be real and positive.  This formula appears as entry $3.466.1$ in 
\cite{gradshteyn-2015a} with value 
\begin{equation}
\label{real-value}
I(a,b) = \frac{\pi}{2b} e^{a^{2}b^{2} }( 1 - \textnormal{Erf}(ab) )
\end{equation}
\noindent
 and the reader will find in \cite{albano-2011a} an elementary proof of it.  This section presents $4$ different ways to 
 use the method of brackets to evaluate this integral. 
 
\smallskip 

\noindent
\texttt{Method 1}.  Start with the bracket series representations
\begin{eqnarray}
\exp(- a^{2} x^{2}) & = & \sum_{n_{1}} \phi_{n_{1}} a^{2n_{1}} x^{2n_{1}} \label{flex-4} \\
\frac{1}{x^{2}+b^{2}} & = & \sum_{n_{2}} \sum_{n_{3}} \phi_{n_{2}n_{3}}  b^{2n_{2}}  x^{2n_{3}}  \langle 1 + n_{2} + n_{3} \rangle \nonumber 
\end{eqnarray}
\noindent
and produce the  bracket series 
\begin{equation}
I(a,b) = \sum_{n_{1}} \sum_{n_{2}} \sum_{n_{3}} \phi_{n_{1}n_{2}+n_{3}} a^{2n_{1}} b^{2n_{2}} 
\langle 1+ n_{2} + n_{3} \rangle \, 
\langle 2n_{1} + 2n_{3} + 1 \rangle.
\end{equation}

\smallskip 

\noindent
\texttt{Method 2}.  The second form begins with the Mellin-Barnes representations 
\begin{eqnarray}
\exp(-a^{2} x^{2} ) & = & \frac{1}{4 \pi i } \int_{\gamma} x^{-t} a^{-t} \Gamma \left( \frac{t}{2} \right) \, dt \label{flex-8} \\
\frac{1}{x^{2}+b^{2}} & = & \frac{1}{4 \pi b^{2} i} \int_{\gamma} x^{-s} b^{s} \Gamma \left( \frac{s}{2} \right) \Gamma \left( 1 - \frac{s}{2} \right) \, ds
\nonumber 
\end{eqnarray}
\noindent
Replacing in \eqref{flex-0} a bracket appears and one obtains the representation 
\begin{equation}
I(a,b) = \frac{1}{4b^{2} (2 \pi i )^{2}} \int_{\gamma_{1}} \int_{\gamma_{2}} a^{-t} b^{s} 
\Gamma \left( \tfrac{t}{2} \right) \Gamma \left( \tfrac{s}{2} \right) \Gamma \left( 1 - \tfrac{s}{2} \right) \langle 1 - t - s \rangle \, ds \, dt.
\label{flex-11}
\end{equation}
\noindent
The bracket is used to evaluate the $s$-integral to produce $s^{*} = 1-t$ and the expression 
\begin{equation}
I(a,b) = \frac{1}{4b^{2} (2 \pi i )} \int_{\gamma} a^{-t} b^{1-t} \Gamma \left( \tfrac{t}{2} \right) 
\Gamma \left( \tfrac{1-t}{2} \right) 
\Gamma \left( \tfrac{1+t}{2} \right)  \, dt.
\end{equation}
\noindent
The next step is to express the gamma factors in the numerator as bracket series to obtain 
\begin{equation*}
I(a,b) = \frac{1}{4b^{2} (2 \pi i )} \sum_{n_{1}} \sum_{n_{2}} \sum_{n_{3}} \phi_{n_{1} n_{2} n_{3}} 
\int_{\gamma} a^{-t} b^{1-t} 
\langle    \tfrac{t}{2} + n_{1} \rangle 
\langle \tfrac{1}{2} - \tfrac{t}{2} + n_{2} \rangle 
\langle \tfrac{1}{2} + \tfrac{t}{2} + n_{3} \rangle \, dt
\end{equation*}
\noindent
and for instance selecting the bracket in $n_{1}$ to eliminate the integral and using 
\begin{equation}
\left \langle \tfrac{t}{2} + n_{1} \right \rangle = 2 \langle t + 2 n_{1} \rangle
\end{equation}
\noindent
gives 
\begin{equation}
I(a,b) = \frac{1}{2b} \sum_{n_{1}} \sum_{n_{2}} \sum_{n_{3}} \phi_{n_{1}n_{2}n_{3}} a^{2n_{1}} b^{2n_{1}} 
\langle \tfrac{1}{2} + n_{1} + n_{2} \rangle 
\langle \tfrac{1}{2} -  n_{1} + n_{3} \rangle.
\end{equation}

\smallskip 

\noindent
\texttt{Method 3}.  The next way to evaluate the integral \eqref{flex-0} is a mixture of the previous two. 
Using 
\begin{eqnarray}
\frac{1}{x^{2} + b^{2}} & = & \sum_{n_{1}} \sum_{n_{2}} \phi_{n_{1}n_{2}} b^{2n_{1}} x^{2n_{2}} \langle 1 + n_{1} + n_{2} \rangle \label{flex-17} \\
\exp(-a^{2}x^{2}) & = & \frac{1}{4 \pi i } \int_{\gamma} x^{-t} a^{-t} \Gamma \left( \tfrac{t}{2} \right) \, dt \nonumber
\end{eqnarray}
\noindent
and the usual procedures gives
\begin{equation*}
I(a,b) = \frac{1}{4 \pi i } \sum_{n_{1}} \sum_{n_{2}} b^{2n_{1}} \langle 1+ n_{1} + n_{2} \rangle \int_{\gamma} a^{-t} 
\Gamma \left( \tfrac{t}{2} \right) \langle 2n_{2} - t +1 \rangle \, dt
\end{equation*}
\noindent
Now use the bracket in $n_{2}$ to eliminate the integral and produce 
\begin{equation}
I(a,b) = \frac{1}{2} \sum_{n_{1}} \sum_{n_{2}} \phi_{n_{1}n_{2}} a^{-2 n_{2}-1} b^{2n_{1}} \langle 1+ n_{1} + n_{2} \rangle \Gamma \left( n_{2} + 
\tfrac{1}{2} \right)
\end{equation}
\noindent
and replacing the gamma factor by its bracket series yields 
\begin{equation}
I(a,b) = \frac{1}{2} \sum_{n_{1}} \sum_{n_{2}} \sum_{n_{3}} \phi_{n_{1}n_{2}n_{3}} b^{2n_{1}} a^{-2n_{2}-1} 
\langle 1+ n_{1} + n_{2} \rangle \langle\tfrac{1}{2} + n_{2} + n_{3} \rangle.
\end{equation}

\smallskip 

\noindent
\texttt{Method 4}. The final form uses the representations 
\begin{eqnarray}
\frac{1}{x^{2} + b^{2}} & = & \frac{1}{4 \pi b^{2}i} \int_{\gamma} x^{-s} b^{s} \Gamma \left( \frac{s}{2} \right) 
\Gamma \left( 1 - \frac{s}{2} \right) \, ds \label{felx-22} \\
\exp(- a^{2}x^{2} ) & = & \sum_{n_{1}} \phi_{n_{1}} a^{2n_{1}} x^{2n_{1}} \nonumber 
\end{eqnarray}
\noindent
to write 
\begin{equation}
I(a,b) = \frac{1}{4 b^{2} \pi i } \sum_{n_{1}}  \phi_{n_{1}} a^{2n_{1}} \int_{\gamma} b^{s} 
\Gamma \left( \tfrac{s}{2} \right) \Gamma \left( 1 - \tfrac{s}{2} \right) \langle 2n_{1} - s + 1 \rangle \, ds,
\end{equation}
\noindent 
where the last bracket comes the original integral. 
Now write the gamma factors as bracket series to produce 
\begin{equation*}
I(a,b) = \frac{1}{2b^{2}(2 \pi i )} \sum_{n_{1}} \sum_{n_{2}} \sum_{n_{3}} \phi_{n_{1} n_{2} n_{3}} a^{2n_{1}} 
\left[ \int_{\gamma} b^{s} \langle \tfrac{s}{2} + n_{2} \rangle \langle 1 - \tfrac{s}{2} + n_{3} \rangle \langle 2n_{1} - s +1 \rangle 
\, ds \right],
\end{equation*}
\noindent
and using the $n_{2}$-bracket to eliminate the integral yields 
\begin{equation}
I(a,b) = \frac{1}{b^{2}} \sum_{n_{1}} \sum_{n_{2}} \sum_{n_{3}} \phi_{n_{1}n_{2}n_{3}} a^{2n_{1}} b^{-2n_{2}} 
\langle 1+ n_{2} + n_{3} \rangle \, \langle 2n_{1} + 2n_{2} + 1 \rangle. 
\label{flex-26}
\end{equation}

These four bracket series lead to the terms 
\begin{eqnarray}
T_{1} & = & \frac{1}{2b} \sum_{k=0}^{\infty} \frac{1}{k!} \Gamma( \tfrac{1}{2} - k ) \Gamma( \tfrac{1}{2} +k ) (-a^{2} b^{2})^{k} \\
& = & \frac{\pi}{2b} \exp(a^{2}b^{2}) \nonumber \\ 
T_{2} & = & \frac{a}{2} \sum_{k=0}^{\infty} \frac{1}{k!} \Gamma(1+k) \Gamma( - \tfrac{1}{2} - k) (-a^{2} b^{2} )^{k} \\
& = & - \frac{\pi}{2b} \exp(a^{2}b^{2}) \textnormal{erf}(ab) \nonumber \\ 
T_{3} & = & \frac{1}{2a b^{2}} \sum_{k=0}^{\infty} \frac{1}{k!} \Gamma(1+k) \Gamma( \tfrac{1}{2} + k ) \left( - \frac{1}{a^{2}b^{2}} \right)^{k} \\
& = & \frac{\sqrt{\pi}}{2ab^{2}}  \pFq20{1 \,\,\,\,  \tfrac{1}{2}}{-}{- \frac{1}{a^{2}b^{2}}}. \nonumber 
\end{eqnarray}
\noindent
The values $T_{1}$ and $T_{2}$ and Rule $E_{3}$ in Section \ref{sec-rules} are combined to give the value 
\begin{equation}
I(a,b) = \frac{\pi}{2b} \exp(a^{2}b^{2}) \left[ 1 - \textnormal{erf}(ab) \right],
\end{equation}
\noindent
conforming \eqref{real-value}. The value of $T_{3}$ is useful in the asymptotic study of $I(a,b)$ as $a^{2}b^{2} \rightarrow \infty$.

\medskip 

The final example illustrates a different combination ideas. 

\begin{example}
Consider the two-dimensional integral 
\begin{equation}
J(a,b) = \int_{0}^{\infty} \int_{0}^{\infty} x^{a-1} y^{b-1} \textnormal{Ei}(-x^{2}y) K_{0} \left( \frac{x}{y} \right) \, dx \, dy
\end{equation}
\noindent
where $\textnormal{Ei}$ is the exponential integral defined by
\begin{equation}
\textnormal{Ei}(-x) =  - \int_{1}^{\infty} \frac{\exp(- x t )}{t} \, dt \quad \textnormal{for } x>0
\end{equation}
\noindent
with divergent representation 
\begin{equation}
\textnormal{Ei}(-x^{2}y) = \sum_{\ell \geq 0} \phi_{\ell} \frac{x^{2 \ell} y^{\ell}}{\ell}
\end{equation}
\noindent
(see \cite{gonzalez-2017a}  for the concept of divergent expansion in the context of the 
method of brackets). Using this series and the Mellin-Barnes representation of $K_{0}$
\begin{equation}
K_{0} \left( \frac{x}{y} \right)  = \frac{1}{4 \pi i } \int_{\gamma} \Gamma^{2}(t) \left( \frac{x}{2y} \right)^{-2t} \, dt
\end{equation}
\noindent
yields 
\begin{equation}
J(a,b) = \frac{1}{4 \pi i } \sum_{\ell \geq 0} \phi_{\ell} \frac{1}{\ell} \int_{\gamma} \Gamma^{2}(t)  2^{2t} 
\langle a + 2 \ell - 2t \rangle \langle b + \ell + 2t \rangle \, dt,
\end{equation}
\noindent
where the two brackets come from the integrals on the half-line. To evaluate this integral use the bracket $\langle a  + 2 \ell - 2 t \rangle $ to produce 
\begin{eqnarray}
\label{1025}
J(a,b) & = &  \frac{1}{8 \pi i } \sum_{\ell \geq 0} \frac{\phi_{\ell}}{\ell} 
\left[ \int_{\gamma} \Gamma^{2}(t) 2^{2t} \langle \tfrac{a}{2} + \ell - t \rangle \langle b + \ell + 2t \rangle \, dt \right] \\
& = & \frac{1}{8 \pi i } \sum_{\ell \geq 0} \frac{\phi_{\ell}}{\ell} \left[ 2 \pi i \Gamma^{2}(t) 2^{2t} \langle  b+ \ell + 2 t \rangle \right]_{t = \tfrac{a}{2} + \ell} 
\nonumber \\
& = & \frac{1}{4} \sum_{\ell \geq 0} \frac{\phi_{\ell}}{\ell} \Gamma^{2} \left( \tfrac{a}{2}+ \ell \right) 2^{a + 2 \ell} \langle a + b + 3 \ell \rangle, \nonumber 
\end{eqnarray}
\noindent
and eliminating the series with the bracket yields the value
\begin{equation}
J(a,b) = - \frac{4^{-1 - b/3 + a/6}}{a+b} \Gamma^{2} \left( \frac{a - 2 b }{6} \right) \Gamma \left( \frac{a+b}{3} \right).
\end{equation}
\end{example}

\section{Feynman diagrams}
\label{sec-feynman}

A variety of interesting integrals appear in the evaluation of Feynman diagrams. See 
\cite{smirnov-2004a,smirnov-2006a} for details. The example corresponds to a loop diagram with a 
propagator in the scalar theory with internal propagators of equal masses \cite{boos-1991a}:
{{
\begin{figure}[!ht]
\begin{center}
\centerline{\psfig{file=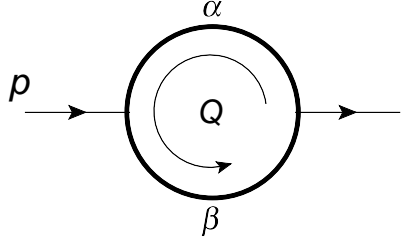,width=15em,angle=0}}
\caption{The loop diagram}
\label{figure1}
\end{center}
\end{figure}
}}
\newline
\noindent
The integral attached to this diagram is 
\begin{equation}
J(\alpha, \beta, m,m) = \int d^{D}Q \frac{1}{ \left[ Q^{2}-m^{2} \right]^{\alpha} \left[ (p+Q)^{2} - m^{2} \right]^{\beta}}
\end{equation}
\noindent
with Mellin-Barnes representation given by 
\begin{multline}
J(\alpha,\beta,m,m) = \pi^{D/2} i^{1-D} \frac{ (-m^{2})^{D/2-\alpha-\beta}}{\Gamma(\alpha) 
\Gamma(\beta)} \\
\times \frac{1}{2 \pi i } \int_{\gamma} 
(p^{2})^{_u} (-m^{2})^{u} 
\frac{\Gamma(-u) \Gamma(\alpha + u) \Gamma(\beta + u) \Gamma(\alpha + \beta - D/2 +u)}
{\Gamma(\alpha + \beta + 2u)} \, du 
\end{multline}
\noindent
Replacing the Gamma factors in the numerator by its corresponding bracket series gives 
\begin{multline}
J(\alpha, \beta, m,m) \leftarrow \pi^{D/2} i^{1-D} 
\frac{(-m^{2})^{D/2-\alpha - \beta}}{\Gamma(\alpha) \Gamma(\beta)} 
\sum_{\{ n \}} \phi_{n_{1}, \ldots, n_{d}}  \\
\frac{1}{2 \pi i} \int_{\gamma} (p^{2})^{u} (- m^{2})^{u} 
\frac{ \langle -u +n_{1} \rangle
 \langle \alpha + u + n_{2} \rangle 
\langle \beta + u + n_{3} \rangle 
\langle \alpha + \beta - D/2+ u + n_{4} \rangle 
} 
{\Gamma(\alpha + \beta + 2u) } \, du.
\end{multline}
\noindent
To evaluate the integral, the vanishing of the bracket 
$\langle - u + n_{1} \rangle$ gives $u^{*} = n_{1}$
and $J$ is expressed as a bracket series
\begin{multline}
J(\alpha, \beta, m,m) = \pi^{D/2} i^{1-D} 
\frac{(-m^{2})^{D/2- \alpha - \beta}}{\Gamma(\alpha) \Gamma(\beta)} \\
\times \sum_{\{ n \}} \phi_{n_{1}, \ldots, n_{4}} 
(p^{2})^{u^{*}} (-m^{2})^{u^{*}} 
\frac{
\langle \alpha + u^{*} + n_{2} \rangle 
\langle \beta + u^{*} + n_{3} \rangle 
\langle \alpha + \beta  - D/2 + u^{*} + n_{4} \rangle 
}
{ \Gamma(\alpha + \beta + 2 u^{*}) }. 
\end{multline}

The terms obtained from the bracket series above are given as  hypergeometric values 
\begin{eqnarray*}
T_{1} & = & \pi^{D/2} i^{1-D} (-m^{2})^{D/2 - \alpha - \beta} 
\frac{\Gamma(\alpha + \beta - D/2)}{\Gamma(\alpha + \beta)} 
\pFq32{\alpha \,\,\,\,\,\,  \beta \,\,\,\,\,\,  \alpha + \beta - \tfrac{D}{2}}{\tfrac{\alpha+ \beta}{2} \,\,\,\,\,\, 
\tfrac{\alpha + \beta +1}{2} }{ \frac{p^{2}}{4m^{2}}}  \\
T_{2} & = & \pi^{D/2} i^{1-D} (-m^{2})^{D/2  - \beta}  (p^{2})^{-\alpha} \, 
\frac{\Gamma( \beta - D/2)}{\Gamma( \beta)} 
\pFq32{\alpha \,\,\,\,\,\,  1+ \frac{\alpha - \beta}{2} 
 \,\,\,\,\,\,  \tfrac{1+  \alpha - \beta}{2}}{  1 + \alpha - \beta \,\,\,\,\,\, 
1 - \beta + \tfrac{D}{2} }{ \frac{4m^{2}}{p^{2}}}  \\
T_{3} & = & \pi^{D/2} i^{1-D} (-m^{2})^{D/2  - \alpha}  (p^{2})^{-\beta} \, 
\frac{\Gamma( \alpha - D/2)}{\Gamma( \alpha)} 
\pFq32{\beta \,\,\,\,\,\,  1+ \frac{\beta - \alpha}{2} 
 \,\,\,\,\,\,  \tfrac{1+  \beta - \alpha}{2}}{  1 + \beta - \alpha \,\,\,\,\,\, 
1 - \alpha + \tfrac{D}{2} }{ \frac{4m^{2}}{p^{2}}}  \\
T_{4} & = & \pi^{D/2} i^{1-D}  (p^{2})^{D/2 - \alpha -\beta} \, 
\frac{\Gamma(D/2 -  \alpha ) \Gamma( D/2 - \beta) \Gamma( \alpha + \beta - D/2)}
{\Gamma(\alpha) \Gamma(\beta) \Gamma(D- \alpha - \beta)}  \\
& & \quad \quad \quad \quad \quad \quad  \times \pFq32{  \alpha + \beta -  \frac{D}{2} 
 \,\,\,\,\,\,  1 + \tfrac{\alpha +  \beta  - D }{2} \,\,\,\,\,  \tfrac{1 + \alpha + \beta -D}{2} }
 {
1 +\beta -   \tfrac{D}{2} \,\,\,\,\, 1 + \alpha - \tfrac{D}{2} }{ \frac{4m^{2}}{p^{2}}}.
\end{eqnarray*}

The conclusion is the evaluation of the Feynman diagram shown in Figure \ref{figure1} is given by 
\begin{equation}
J(\alpha, \beta, m, m ) = 
\begin{cases}
T_{1} \quad & \textnormal{for} \,\, \left| \frac{p^{2}}{4m^{2}}   \right| < 1 \\
T_{2}  + T_{3} + T_{4} \quad & \textnormal{for} \,\, \left| \frac{4m^{2}}{p^{2}}  \right| < 1. 
\end{cases}
\end{equation}

. 

\section{Conclusions}
\label{sec-conclusion}

The work presented here gives a new procedure to evaluate Mellin-Barnes integral representations. The 
main idea is to replace the Gamma factors appearing in such representations by its corresponding 
bracket series.  The method of brackets is then used to evaluate these integrals. This method has been 
shown to be simple and efficient, given its  algorithmic nature.  A collection of representative examples 
has been provided.

\section{Acknowledgments} 
I.K. was supported in part by Fondecyt (Chile) Grants Nos. 1040368, 1050512 and 1121030, by DIUBB (Chile) Grant Nos. 102609,  GI 153209/C  and GI 152606/VC.  I.~G. wishes to thank the hospitality of the
 Department of Mathematics at Tulane University, where part of this 
work was conducted. 

\smallskip 

The Feynman diagram has been drawn using JaxoDraw \cite{binosi-2009a}.

\end{document}